\documentclass[a4paper,11pt,reqno]{amsart}%,oneside

\usepackage[ansinew]{inputenc}
\usepackage[english]{babel}
\usepackage{amsmath}
\usepackage{amssymb}
\usepackage{amsthm}
\usepackage{mathrsfs}
\usepackage[all]{xy}
\usepackage{verbatim}
\usepackage{arydshln}
\usepackage{multicol}
\usepackage{enumitem}
\usepackage{bm}
\usepackage{bbm}

\usepackage[bookmarks=true,bookmarksopen=false]{hyperref}
\hypersetup{
   pdftitle = {},
   pdfauthor = {Matthew Gelvin},
   pdfpagemode = {UseOutlines},
   pdfstartview = {FitH},
   pdfborder = {0 0 0},
   backref = {true},
   colorlinks = {true},
   urlcolor = {blue},
   citecolor = {purple},
   linkcolor = {blue},
   pdftoolbar = {true},
}

\usepackage{mathtools}
\usepackage{ifthen}
\usepackage{tikz}
\usepackage{todonotes}
\usepackage{bookmark}

\usepackage[abbrev,msc-links]{amsrefs}

\numberwithin{equation}{section}
\numberwithin{figure}{section}
\numberwithin{table}{section}

\allowdisplaybreaks

\theoremstyle{plain}
\newtheorem{theorem}{Theorem}[section]
\newtheorem{proposition}[theorem]{Proposition}
\newtheorem{lemma}[theorem]{Lemma}
\newtheorem{corollary}[theorem]{Corollary}

\theoremstyle{definition}
\newtheorem{definition}[theorem]{Definition}

\theoremstyle{remark}
\newtheorem{remark}[theorem]{Remark}

\newtheorem{example}[theorem]{Example}

\DeclareMathOperator{\Aut}{Aut}

\DeclareMathOperator{\End}{End}

\DeclareMathOperator{\Hom}{Hom}

\DeclareMathOperator{\Map}{Map}

\DeclareMathOperator{\Syl}{Syl}
\DeclareMathOperator{\res}{Res}
\DeclareMathOperator{\ind}{Ind}
\DeclareMathOperator{\ten}{Ten}
\DeclareMathOperator{\jnd}{Jnd}

\DeclareMathOperator{\OH}{\mathcal{O} _{\cH}}

\newcommand{\un}{\underline}

\newcommand{\Hn}{\mathfrak{Hom}}
\newcommand{\bd}{\partial}
\newcommand{\brres}{\mathrm{BrRes}}  

\newcommand{\id}{\mathrm{id}}

\newcommand{\Dim}{\mathrm{Dim}}

\newcommand{\f}{\mathfrak{f}}

\def\cF{\mathcal F}\def\cH{\mathcal H}
\def\cK{\mathcal K}
\def\cM{\mathcal M}\def\cO{\mathcal O}\def\cP{\mathcal P}

\def\bD{\mathbf D}

\def\CC{\mathbb C}
\def\FF{\mathbb F}
\def\KK{\mathbb K}

\def\RR{\mathbb R}

\def\ZZ{\mathbb Z}

\def\id{\mathrm{id}}

\makeatletter
\renewcommand*\env@matrix[1][c]{\hskip -\arraycolsep
  \let\@ifnextchar\new@ifnextchar
  \array{*\c@MaxMatrixCols #1}}
\makeatother

\def\maprt#1{\smash{\,\mathop{\longrightarrow}\limits^{#1}\,}}

\begin{document}

\title{Dade Groups for Finite Groups and Dimension Functions}
\author{Matthew Gelvin and Erg\"un Yal\c c\i n}

\date{\today}
 
\address{Department of Mathematics, Bilkent University, 06800 
Bilkent, Ankara, Turkey}

\email{matthew.gelvin@bilkent.edu.tr, yalcine@fen.bilkent.edu.tr}

\begin{abstract}  
Let $G$ be a finite group and $k$ an algebraically closed field of characteristic $p>0$.
We define the notion of a Dade $kG$-module as a generalization of endo-permutation modules for $p$-groups.
We show that under a suitable equivalence relation, the set of equivalence classes of Dade $kG$-modules 
forms a group under tensor product, and  the group obtained this way is isomorphic to the Dade 
group $D(G)$ defined by Lassueur.  We also consider the subgroup $D^{\Omega} (G)$ of $D(G)$ generated by relative syzygies $\Omega_X$, 
where $X$ is a finite $G$-set.  If $C(G,p)$ denotes the group of superclass functions defined on the $p$-subgroups 
of $G$, there are natural generators $\omega_X$ of $C(G,p)$, and we prove the existence of a well-defined group 
homomorphism $\Psi_G:C(G,p)\to D^\Omega(G)$ that sends $\omega_X$ to $\Omega_X$. 
The main theorem of the paper is the verification that the subgroup of $C(G,p)$ consisting of the dimension functions 
of $k$-orientable real representations of $G$ lies in the kernel of $\Psi_G$.

%We also consider the subgroup $D^{\Omega} (G)$ of $D(G)$ generated 
%by relative syzygies $\Omega_X$ for $X$ a finite $G$-set, and show that there is a well-defined homomorphism 
%$\Psi_G: C(G,p) \to D^{\Omega } (G)$ from the group of superclass functions $C(G,p)$ defined on $p$-subgroups  
%to the relative syzygy group $D^{\Omega}(G)$ sending the generators $\omega_X$ to $\Omega_X$.
%The main theorem of the paper 
%is the verification that the subgroup of $C(G,p)$ consisting of the dimension functions of $k$-orientable real representations of $G$ 
%lies in the kernel of $\Psi_G$. 
\end{abstract}



\maketitle

%-------------------------------------

\section{Introduction}\label{sect:Introduction}
 
Let $G$ be a finite group with Sylow $p$-subgroup $S$, and let $k$ be an algebraically 
closed field of characteristic $p>0$. Throughout we assume that all $kG$-modules are finitely generated. 
A $kG$-module $M$ is a \emph{permutation module} if it has a $G$-invariant basis. A $kG$-module 
$M$ is called an \emph{endo-permutation module} if $\End(M)\cong M^* \otimes M$ is a permutation $kG$-module. 
Endo-permutation modules play an important role in representation theory, for example, they appear as sources of simple modules.

When $G=S$ is a $p$-group, every permutation $kS$-module $k[S/Q]$ formed by linearizing a transitive $S$-set
is indecomposable. Hence 
a summand of a permutation module is also a permutation module. This together with other properties  
make it possible to define a group of endo-permutation modules. An endo-permutation module is said to be \emph{capped} 
if it has a summand with vertex $S$. Two endo-permutation modules $M$ and $N$ are called \emph{compatible}
if $M\oplus N$ is also an endo-permutation module. Compatibility defines an equivalence relation on endo-permutation modules. 
The \emph{Dade group} $D(S)$ of a $p$-group $S$ is defined 
to be the group whose elements are the equivalence classes of capped endo-permutation $kS$-modules and whose 
group operation is induced by tensor product, i.e., $[M]+[N] :=[M\otimes N]$.  The Dade group of a $p$-group has been studied 
by many authors; a complete description in terms of the genetic sections of the group is given by 
Bouc in \cite{Bouc-Dade} (see also \cite{Dade-Endo}, \cite{CarlsonThevenaz-TorsionEndo},   
\cite{CarlsonThevenaz-Classification}, and  \cite{BoucThevenaz-Endo-permutation}).

When $G$ is a finite group, the situation is more complicated since transitive permutation $kG$-modules need not be 
indecomposable. For finite groups it makes sense to extend 
the definition of endo-permutation $kG$-modules in the following way: A $kG$-module $M$ is a 
\emph{$p$-permutation module} if it is a summand of a permutation $kG$-module. 
This is equivalent to requiring that $\res^G_S M$ be a permutation $kS$-module. 
A $kG$-module $M$ is called an \emph{endo-$p$-permutation module} if 
$\End(M) \cong M ^* \otimes M$ is a $p$-permutation module.

In \cite{Urfer-Endo}, Urfer showed that the source of an indecomposable endo-$p$-permutation module $M$ with vertex $S$  
is a capped endo-permutation $kS$-module and its equivalence class in $D(S)$ 
is $G$-stable (see \cite[Theorem 1.5]{Urfer-Endo}). Two endo-$p$-permutation 
modules are declared equivalent if their corresponding source modules represent the same class in $D(S)$.
Under this equivalence relation, the equivalence classes of capped endo-$p$-permutation modules have a natural 
group structure induced by the tensor product, which yields a group isomorphic 
to the subgroup of $G$-stable elements $D(S)^{G-st}$ in $D(S)$. Unfortunately, this definition of equivalence 
is too coarse; many interesting endo-$p$-permutation modules are identified with each other.
For example, if $S$ is normal in $G$, then any two one-dimensional $kG$-representations with kernel $S$ are 
equivalent.  

Another approach to defining the Dade group of a finite group is given by Lassueur in \cite{Lassueur-Dade}.
There, one considers endo-$p$-permutation modules that are endotrivial relative to the family of non-Sylow $p$-subgroups 
of $G$. Such modules are called \emph{strongly capped}. Every strongly capped $kG$-module $M$ has 
a unique indecomposable summand, called the \emph{cap} of $M$. Two 
strongly capped modules are declared equivalent if their caps are isomorphic. 
Lassueur defines the Dade group $D(G)$ as the group whose elements are the equivalence classes of 
strongly capped endo-$p$-permutation $kG$-modules and whose group operation is induced by the tensor product 
(see \cite[Cor.Def. 5.5]{Lassueur-Dade}). 

Lassueur's definition of the Dade group captures the differences 
between one dimensional representations when the Sylow $p$-subgroup $S$ is normal in $G$. Moreover, the relationship 
between Lassueur's Dade group and Urfer's $G$-stable elements notion is completely understood: Let $\Upsilon(G) \leq D(G)$ denote the subgroup of equivalence classes of strongly capped indecomposable $kG$-modules corresponding to one-dimensional representations of $N_G(S)/S$ under the Green correspondence. Lassueur proves that there is a short exact sequence of abelian groups
 $$0 \to \Upsilon  (G) \to D(G) \to D(S)^{G-st} \to 0$$
where the second map in the sequence is given by restriction to the Dade group of $S$  (see \cite[Theorem 7.3]{Lassueur-Dade}).

As a generalization of Lassueur's definition of strongly capped endo-$p$-permutation module, we define a notion of a Dade $kG$-module. 

\begin{definition}\label{def:DadeModule} A $kG$-module $M$ is a \emph{Dade $kG$-module} if there is an integer $n\geq 0$ such that 
$$\End(M) \cong k^n \oplus W$$ for some $p$-permutation module $W$, all of whose
indecomposable summands have vertices that are non-Sylow $p$-subgroups of $G$. 
\end{definition}

A Dade module is \emph{capped} if it has a Sylow-vertex component, or equivalently if $n \geq 1$ in the above decomposition (see Lemma  \ref{lemma:CappedEquiv}). Two Dade modules $M$ and $N$ are \emph{compatible} if $M\oplus N$ is a Dade module, or equivalently if $M^* \otimes N$ is a Dade module. Being compatible defines an equivalence relation on the set of capped Dade modules, which we denote by $M \sim N$  (see Lemma \ref{lem:EquivRelations}). We show the following: 
 
\begin{theorem}\label{thm:DadeGroup} 
Let $D(G)$ denote the set of equivalence classes of capped Dade $kG$-modules under the equivalence relation defined above.
Then the operation $[M]+[N]:=[M\otimes N]$  defines an abelian group structure on $D(G)$. Moreover, the group $D(G)$ defined 
this way is isomorphic to the Dade group defined by Lassueur in \cite{Lassueur-Dade} using relative endotrivial modules.
\end{theorem}
 
We prove Theorem \ref{thm:DadeGroup} in two parts, as Propositions \ref{pro:DadeGroup} and  \ref{pro:IsomDade}. We also show that every capped Dade module has, up to isomorphism, a unique indecomposable summand with vertex  $S$, called the \emph{cap} of $M$ (see Proposition \ref{pro:UniqueCaps}).  The cap of a Dade module is a strongly capped endotrivial module in the sense of Lassueur. In fact, an endo-$p$-permutation $kG$-module $M$ is strongly capped if and only if $M$ is a capped Dade module with a unique copy of its cap (see Proposition \ref{pro:StronglyCapped}).  

One important source of Dade modules are the $kG$-modules defined as relative syzygies. Given a  $G$-set $X$, the kernel of the augmentation map $\varepsilon: kX\to k$
is called the \emph{relative syzygy of $X$}, and is denoted by $\Delta (X)$. Since the restriction of $\Delta (X)$ to $S$ is a endo-permutation module, $\Delta(X)$ is an endo-$p$-permutation module. Using arguments similar to those used in the $p$-group case, one can show that 
if $X$ is a $G$-set such that $X^S=\emptyset$, then $\Delta(X)$ is a capped Dade module, and hence defines a class $\Omega_X:=[\Delta (X)]$ in $D(G)$.  We extend this definition to all $G$-sets by declaring $\Omega_X=0$ whenever $X^S \neq \emptyset$. 

\begin{definition} The subgroup of $D(G)$ generated by the elements $\Omega _X$ as $X$ ranges over all $G$-sets is called  the \emph{Dade group generated by relative syzygies}, and is denoted by $D^{\Omega} (G)$.  
\end{definition}

Many results known for relative syzygies over $p$-groups  hold for relative syzygies over finite groups as well. We prove these results 
in Section \ref{sect:RelativeSyzygies}. In particular,  we show that $D^{\Omega} (G)$ is generated by relative syzygies of the form 
$\Omega _{G/P}$ where $P \leq G$ is a non-Sylow $p$-subgroup of $G$ (see Proposition \ref{prop:Generates}). 

For a finite group $G$ and a fixed prime $p$, we denote by $\cF_p$ the family of all $p$-subgroups in $G$. 
A function $f : \cF_p \to \ZZ$ that is constant on the $G$-conjugacy classes of subgroups in $\cF_p$ is called a \emph{superclass function}. 
The set of superclass functions defined on $\cF_p$ forms a group under addition. We denote this group by $C(G,p)$. 
For each $G$-set $X$, there is a superclass function $\omega_X$ defined by 
%\[
%\omega _X (P) =\left\{
%\begin{array}{ll}
%1&\textrm{if }X^P\neq\emptyset,\\
%0&\textrm{otherwise.}
%\end{array}
%\right.
%\]
$$\omega _X (P) =\begin{cases} 1 \ \text{  if   }  \ X^P \neq \emptyset, \\ 0 \  \text{ otherwise.} \end{cases} $$
There are relations among the superclass functions $\omega_X$ coming from the
inclusion-exclusion principle. We prove that all relations satisfied by the generators $\omega _X$ in
$C(G, p)$ are already satisfied by the generators $\Omega_X$ of $D^{\Omega} (G)$. This gives us the following:

\begin{theorem}\label{thm:IntroBoucHom} There is a well-defined surjective group homomorphism 
$$\Psi _G : C(G, p) \to D^{\Omega} (G)$$ that sends $\omega_X$ to $\Omega_X$ for every $G$-set $X$. 
\end{theorem}

We prove this theorem in Section \ref{sect:Superclass}. We call the homomorphism $\Psi_G$ the \emph{Bouc homomorphism} for $G$ since it is a generalization of the homomorphism defined by Bouc in \cite{Bouc-Remark} for $p$-groups. 

In the rest of the paper we focus on identifying  the kernel of the homomorphism $\Psi_G$ in terms of the subgroup of dimension functions of $k$-orientable real representations. The \emph{dimension function} of a real representation $V$ is the superclass function $\Dim (V)  : C(G, p)\to \ZZ$ defined by
$$\Dim (V) (P)=\dim _{\RR} (V^P) $$ 
for every $p$-subgroup $P \leq G$. 

For a real $G$-representation $V$, let $S(V)$ denote the unit sphere of $V$ under a $G$-invariant norm.  Then $X:=S(V)$ is a $G$-space, and for every $p$-subgroup $P\leq G$, the fixed point subspace $X^P =S(V^P)$ is homeomorphic to a sphere of dimension $\dim (V^P) -1$.

%The $G$-invariant unit sphere $S(V)$ of a real representation $V$ is a $G$-space, which can be triangulated to obtain a $G$-CW-complex structure. 
%Note that for each $P\in \cF_p$, we have $X^P=S(V^P)$, hence the fixed point subspace $X^P$ is homeomorphic to a sphere of dimension $\dim (V^P) -1$. The \emph{dimension function} of the representation sphere $X=S(V)$ is the function $\Dim (X):  \cF_p \to \ZZ$ defined by $\Dim (X) (P)=\dim (X^P)+1$ for every $P \in \cF_p$.   

\begin{definition}\label{def:kOrientable} A real $G$-representation $V$ is called \emph{$k$-oriented} if  the $\overline N_G(P)=N_G (P)/P$-action on the reduced homology group $\widetilde H_* (S(V)^P, k) \cong k$ is trivial for every $P \in \cF_p$. The set of isomorphism classes of $k$-orientable real representations forms a group under direct sum, which we denote by $R_{\RR } ^{+} (G, k)$. 
\end{definition}

There is a group homomorphism $$\Dim : R^+ _{\RR} (G, k) \to C(G, p)$$ that takes a representation $V$ to its dimension function. The main result of this paper is the following theorem.

\begin{theorem}\label{thm:IntroInclusion} The image of the dimension function $\Dim : R^+ _{\RR} (G, k) \to C(G, p)$ lies in the kernel of Bouc homomorphism $\Psi _G : C(G, p) \to D^{\Omega} (G)$.
\end{theorem}

We prove Theorem \ref{thm:IntroInclusion} in Section \ref{sect:CappedMoore}. The proof relies on topological methods, which require us to consider Moore $G$-spaces. A \emph{Moore $G$-space over $k$ relative to the family $\cF_p$} is a $G$-CW-complex $X$ such that for every $P \in \cF_p$, the fixed-point subspace $X^P$ is a Moore space over $k$, i.e., the reduced homology $\widetilde H_* (X^P; k)$ is nonzero only in a single dimension. The \emph{dimension function} of a Moore $G$-space $X$ is defined to be the superclass function $\Dim (X):  C(G,p) \to \ZZ$ satisfying  $$\Dim (X) (P)=\dim (X^P)+1$$ for every $p$-subgroup $P\in \cF_p$.

If $V$ is a $k$-oriented real $G$-representation, then the unit sphere $X=S(V)$ can be triangulated to obtain a $G$-CW-complex, which is a Moore $G$-space whose homology is the trivial module $k$. Note that the dimension function of $X=S(V)$ as a Moore $G$-space
coincides with the algebraic dimension function $\Dim (V)$ defined before.
 
We show that the reduced homology of an $n$-dimensional Moore $G$-space $X$ whose point-stabilizers are all non-Sylow $p$-subgroups  is a capped Dade  $kG$-module. Moreover, if $X_i$ denotes the $G$-set of $i$-cells of $X$, then we have that the equality $$[\widetilde H_{n} (X; k)]=\sum _{i=1}^n \Omega _{X_i}$$ holds in $D^{\Omega} (G)$ (see Theorem \ref{thm:Moore}).
This gives the equality $[\widetilde H_{n} (X; k)]=\Psi _G(\Dim (X))$ under some minor technical assumptions.
This is analogous to the result proved in \cite{Yalcin-Moore} for Moore $S$-spaces when $S$ is a $p$-group.

We then consider Moore $G$-spaces whose isotropy subgroups are arbitrary, including the possibility that $X^S$ is nonempty.  If $X$ is a Moore space such that $X^S$ is $m$-dimensional, we say $X$ is \emph{capped} if the reduced homology module $\widetilde H _m (X^S ; k)$ has the trivial module $k$ as a summand.  We prove in Theorem \ref{thm:CappedMooreHom} that if $X$ is a capped $n$-dimensional Moore $G$-space, then the reduced homology module $\widetilde H_n (X; k)$ has a summand $M$ that is a Dade module whose class in $D(G)$ is equal to $\sum _{i=1}^n \Omega _{X_i}=\Psi _G( \Dim (X))$. We then apply this result to a $k$-orientable real representation sphere $X=S(V)$ and obtain that $\Psi _G (\Dim (X))=0$ in $D^{\Omega} (G)$. This completes the proof of Theorem \ref{thm:IntroInclusion}  (see the end of Section \ref{sect:CappedMoore}).
 
In the last section of the paper, we study the subgroup of $C(G,p)$ formed by dimension functions of $k$-orientable real representations of $G$. We show that these superclass functions are characterized by a set of conditions defined on certain subquotients of $G$. Some of these conditions come from  the Borel-Smith conditions defined on $p$-subgroups of $G$. We denote the subgroup of $C(G,p)$ satisfying the Borel-Smith conditions on the $p$-subgroups of $G$ by $C_b (G, p)$ (see  Definition \ref{def:BorelSmith}). We also consider some additional conditions called the \emph{oriented Artin conditions} similar to the conditions given by Bauer in \cite{Bauer} for finite groups (see Definition \ref{def:OrientedArtin}). We denote by $C_{ba^+} (G, p)$ the subgroup of $C(G, p)$ satisfying both the Borel-Smith conditions and the oriented Artin conditions. 

By modifying the arguments given by Bauer, we show that the image of the dimension function $\Dim : R_{\RR} ^+ (G) \to C(G,p)$ is equal to the subgroup $C_{ba^+} (G, p)$.  For $p=2$ the oriented Artin conditions always hold by trivial reasons, so in that case we have $C_b(G, p)=C_{ba^+} (G, p)$. This allows us to conclude the following:

\begin{corollary}\label{cor:introInequalities} With the notation introduced above, we have $$C_{ba^+ } (G, p) \subseteq \ker \Psi_G \subseteq C_b(G, p).$$
Moreover, when $p=2$, all the subgroups above are equal and we have a short exact sequence of abelian groups
$$0 \to C_b (G, 2) \maprt{j} C(G, 2) \maprt{\Psi_G}  D^{\Omega} (G) \to 0.$$ 
\end{corollary}

It is an open question whether the equality $C_{ba^+ } (G, p) = \ker \Psi_G$ also holds when $p$ is odd. This question is related to the existence of well-defined restriction and deflation maps on $D^{\Omega} (-)$ that commute with the homomorphism $\Psi_{(-)}$. If such maps exist, then one can prove the equality $C_{ba^+ } (G, p)=\ker \Psi _G$ using an inductive argument and by applying Lemma \ref{lem:SemiDirect}. 
Conversely, the existence of an exact sequence as the one given in Corollary \ref{cor:introInequalities} gives a biset functor structure to $D^{\Omega} (-)$, in a modified sense that allows only deflations with normal $p$-subgroups.  

{\bf Acknowledgement:}  The second author is supported by a T\" ubitak 1001 project (grant no: 116F194).

\setcounter{tocdepth}{1}
\tableofcontents
%\tableofcontents[hideallsubsections]

%-------------------------- 

\section{Dade modules}\label{sect:DadeModules}

Throughout this paper, let $G$ be a finite group and $p$ a prime number, not necessarily dividing $|G|$.   
We write $S$ for a Sylow $p$-subgroup of $G$ and $P, Q,\ldots$ for arbitrary $p$-subgroups. 
Let $k$ be a fixed algebraically closed field of characteristic $p$. 

Our aim in this section is to give the definition and basic properties of Dade $kG$-modules.  
We assume that all modules are finitely generated.  For a $kG$-module $M$, we write $M^\ast$ for 
the dual $kG$-module $\Hom_k(M,k)$. When $N$ is a summand of $M$ we use the standard notation 
$N \mid M$ to express this relation. An indecomposable direct summand of a $kG$-module $M$ is called 
a \emph{component} of $M$.

We recall the following notion: The $kG$-module $M$ is called a \emph{permutation $kG$-module} if $M$ has a 
$G$-invariant $k$-basis $X$, i.e., $g \cdot X=X$ for all $g\in G$.  Any permutation $kG$-module can be written
\begin{equation}\label{eqn:permutation}
M\cong\bigoplus_{[H\leq G]} c_H \cdot k[G/H],
\end{equation}
where the sum is taken over the $G$-conjugacy classes of subgroups of $G$ and each $c_H$ is some 
non-negative integer.  

If $G=S$ is a $p$-group, Green's Indecomposability Criterion 
implies that a transitive permutation $S$-module is indecomposable (see \cite[Cor 11.6.3]{Webb-Book}). Hence in this case 
the right hand side of Equation (\ref{eqn:permutation}) is uniquely determined by the left. In particular, any direct summand 
of a permutation $kS$-module is also a permutation $kS$-module.

For arbitrary finite groups none of these properties needs to hold, so we turn to a more natural class of a module.  
 
\begin{definition}\label{def:pPermutation} A $k G$-module $M$ is called a \emph{$p$-permutation module} if any of the following equivalent conditions is satisfied:
\begin{enumerate}
\item The restriction $\res ^G _S M $ is a permutation $kS$-module.
\item $M$  is a direct summand of a permutation module 
of the form $$\bigoplus_{[P \leq G]} c_P\cdot k[G/P],$$ where the sum is taken over the $G$-conjugacy classes of 
$p$-subgroups of $G$.  
\item $M$ is a direct summand of a permutation module.
\end{enumerate} 
\end{definition}
 
The equivalence of the above conditions is easy to show. The main observation is that a $kG$-module $M$ is a direct 
summand of the module $\ind _S ^G \res ^G _SM$, which gives the implication $(i) \implies (ii)$.  The other implications 
are obvious (see \cite[Lemma 3.11.2]{Benson-Book1} for details).  

Given an indecomposable $kG$-module $M$, there is a subgroup $H \leq G$ that is minimal with 
the property that $M$ is a summand of $\ind ^G  _H \res^G _H M$. Such a subgroup is called a 
\emph{vertex of $M$}. It is easy to see that a vertex of $M$ is a $p$-subgroup of $G$ and that 
any two vertices are $G$-conjugate. 

If $M$ is an indecomposable $kG$-module with vertex $P$, then there is 
an indecomposable $kP$-module $U$ such that $M$ is a direct summand of $\ind _P ^G U$. 
This module is unique up to isomorphism and conjugation by elements in $N_G(P)$, and 
is called the \emph{source} of $M$ (see \cite[Thm 11.6.1]{Webb-Book}). 

From Definition \ref{def:pPermutation}, it is easy to see that an indecomposable $kG$-module $M$ is a $p$-permutation $kG$-module 
if and only if a vertex-source pair of $M$ is of the form $(P,k)$, so that the source is the trivial $kP$-module.  
A $kG$-module is called a \emph{trivial source module} if all of its components have a
trivial source. A $kG$-module $M$ is a $p$-permutation module if and only if it is 
a trivial source module.

\begin{definition} Let $M$ be a $p$-permutation $kG$-module. A component of $M$ whose vertex is a Sylow $p$-subgroup 
of $G$ is called a \emph{Sylow-vertex component} of $M$.  We call a $p$-permutation module \emph{Sylow-vertex trivial} if all its Sylow-vertex components, should they exist, 
are isomorphic to $k$.
\end{definition}

Note that a Sylow-vertex trivial $p$-permutation $kG$-module is isomorphic to a direct sum of the form $k^n \oplus W$, where $n \geq 0$ is a natural number and  $W$ is a $p$-permutation module whose components all have non-Sylow vertices. It is easy to see the following:

\begin{lemma}\label{lem:Closedness} The class of Sylow-vertex trivial  $p$-permutation $kG$-modules is closed under taking direct sums, 
direct summands, and tensor products.
\end{lemma}

We now expand from permutation and $p$-permutation $kG$ modules to a larger class.  Recall that the $kG$-module $M$ is an \emph{endo-permutation $kG$-module} if $\End(M)\cong M^* \otimes M$ is a permutation $kG$-module via the natural diagonal action.  As above, this definition really works best when $G=S$ is a $p$-group, but  it has the following natural extension: We say that $M$ is an \emph{endo-$p$-permutation $kG$-module} if $\End(M)$ is a $p$-permutation $kG$-module, or equivalently if $\res ^G_S(M)$ is an endo-permutation $kS$-module.  
 
In this paper we consider endo-$p$-permutation $kG$-modules that satisfy an extra property.  A $kG$-module $M$ is called a  
\emph{Dade $kG$-module} if
$$\End(M) \cong k^n \oplus W$$ for some $n\geq 0$ and for some $p$-permutation module $W$, 
all of whose components have non-Sylow vertices (see Definition \ref{def:DadeModule}). In other words, a $kG$-module $M$ 
is a Dade module if $\End(M)\cong M^* \otimes M$ is a Sylow-vertex trivial $p$-permutation module.   
By definition a Dade module is an endo-$p$-permutation module, but not every endo-$p$-permutation module is a Dade module.  

\begin{example} Consider the symmetric group $G=\Sigma_3$ at the prime $p=3$. The Sylow $3$-subgroup $S$ is isomorphic to $C_3$, the cyclic group of order $3$. Let $M=k[G/S]$ be the permutation $kG$-module whose basis is the transitive $G$-set $G/S$. 
Then $M$ is a $p$-permutation module, hence an endo-$p$-permutation module. 
However, $M$ is not a Dade module since
 $$\End(M) \cong M ^* \otimes M  \cong 2k[G/S] \cong 2k\oplus 2 k(-1),$$ 
where $k(-1)$ denotes the alternating representation of $G$. 
Note that the component $k(-1)$ of $\End(M)$ has vertex $S$ but is not trivial. 
\end{example}

We also have a notion of a capped Dade module.  
 
\begin{definition}\label{def:Capped}  A Dade $kG$-module $M$ is \emph{capped} if it has a Sylow-vertex component.
\end{definition}

There are equivalent ways to say a Dade module is capped.  

\begin{lemma}\label{lemma:CappedEquiv} Let $M$ be a Dade $kG$-module. Then the following are equivalent:
\begin{enumerate}
\item\label{lemma:CappedEquiv_i} $M$ is capped, i.e., $M$ has a component with vertex $S$.
\item\label{lemma:CappedEquiv_ii} $\res^G _S M$ is a capped endo-permutation $kS$-module. 
\item\label{lemma:CappedEquiv_iii} $k$ is a summand of $\End(M)$. 
\end{enumerate}

\begin{proof} If $M$ has a component with vertex $S$, then $\res^G _S M$ also has a component with vertex $S$, so $\res^G_S M$ is a capped endo-permutation module. This gives $\ref{lemma:CappedEquiv_i} \Rightarrow \ref{lemma:CappedEquiv_ii}$. 

Now assume $\res^G _S M$ is a capped endo-permutation $kS$-module. Then $\res ^G _S (\End (M))=\End (\res^G _S M)$ has $k$ as a summand (see \cite[Lemma 12.2.6]{Bouc-BisetBook}).
Since $\End (M) \cong k^n \oplus W$ where $W$ has no components with vertex $S$, we see that we must have $n \geq 1$ in order to have that $k$ is a summand of $\res^G _S (\End (M))$. This gives $\ref{lemma:CappedEquiv_ii}\Rightarrow\ref{lemma:CappedEquiv_iii}$. 

To show the final implication $\ref{lemma:CappedEquiv_iii}\Rightarrow \ref{lemma:CappedEquiv_i}$, observe that if all components of $M$ have vertices strictly smaller than $S$, then all components of $\End(M) \cong M^* \otimes M$ have vertices strictly smaller than $S$ as well, so that $k$ cannot be a summand of $\End(M)$.
\end{proof}
\end{lemma}

An indecomposable summand of a capped Dade module $M$ is called a \emph{cap} of $M$.  
In Section \ref{sect:Indecomposable} we show that the cap of a Dade $kG$-module $M$ is unique 
up to isomorphism (Theorem \ref{pro:UniqueCaps}), so we will write $Cap (M)$ for the isomorphism class of the cap of $M$ without confusion.

We note that the property of being a Dade module is detected on the normalizer of a Sylow $p$-subgroup.

\begin{lemma}\label{lem:Detection} A $kG$-module $M$ is a (capped) Dade module if and only if $\res^G _{N_G(S)} M$ is a (capped) 
Dade $kN_G(S)$-module.
\end{lemma}

\begin{proof}  The ``only if"  assertion is trivial. For the other direction, assume that $\res^G _{N_G(S)} M$ is a Dade $kN_G(S)$-module. Then $$\res ^G _S M=\res^{N_G(S)} _S \res^G _{N_G(S)} M$$ is an endo-permutation $kS$-module, and hence $M$ is an endo-$p$-permutation module. Let $U$ be 
a Sylow-vertex component of  $\End (M)$ (if there is any). Then $\res^G _{N_G (S) } U$ is a Sylow-vertex component of
$$\res^G _{N_G(S) } ( \End (M) ) \cong \End ( \res^G _{N_G(S)}  M ) \cong k^n \oplus W,$$ 
where all the components of the $N_G(S)$-module $W$ have non-Sylow vertices. This gives that $\res^G _{N_G(S) } U \cong k$. By the Green correspondence, we obtain that $U \cong k$ as a $kG$-module.  This proves that $M$ is a Dade $kG$-module.
\end{proof}

We now give a few examples of capped Dade modules.

\begin{example}
If $G=S$ is a $p$-group, then the notions of endo-permutation $kS$-module and Dade $kS$-module coincide, and the adjective ``capped'' is consistent with this identification. More generally, Lemma \ref{lem:Detection} implies that if $S$ is self-normalizing in $G$, then every endo-$p$-permutation $kG$-module is a  Dade module.
\end{example}
 
\begin{example}
Every one-dimensional character $\chi:G \to k$ is a capped Dade $kG$-module, as $\End(\chi)\cong\chi^\ast\otimes\chi\cong k$. More generally, if $M$ is an indecomposable $kG$-module whose Green correspondent is a one-dimensional $kN_G(S)$-module, then by Lemma \ref{lem:Detection}, $M$ is a capped Dade module. The Dade modules obtained this way generate a subgroup of the Dade group denoted by $\Upsilon (G)$, which is equal to the kernel of the restriction map  from the Dade group of $G$ to the Dade group of $S$ (see Theorem \ref{thm:LassueurExact}).
\end{example}

\begin{example}  Let $X$ be a $G$-set  such that $X^S =\emptyset$, and let $\Delta (X)$ denote the kernel of the augmentation map $\varepsilon: kX\to k$ defined by $\varepsilon (x)=1$ for all $x\in X$.  Then $\Delta (X)$ is a capped Dade module (see Proposition \ref{pro:RelativeDade}). 
\end{example}
 
\begin{example} A $kG$-module $M$ is called \emph{endotrivial} if $M^* \otimes M \cong k\oplus (proj)$. An endotrivial $kG$-module is a Dade module since a projective module is easily seen to be a $p$-permutation module. In particular, if $M$ is a Sylow-trivial module in the sense of \cite{Grodal}, i.e., if $\res^G _S M \cong k \oplus (proj)$,
then $M$ is an endotrivial module, and hence also a Dade module (see \cite[Thm 2.2]{Carlson-Survey}).
\end{example}

It is immediate that the class of Dade modules is closed under taking a direct summand and tensor products, although not in general under taking direct sums (just as the direct sum of two endo-permutation $kS$-modules need not be endo-permutation). We use this to define an equivalence relation on Dade modules.

\begin{definition}\label{def:Compatible}
The Dade $kG$-modules $M$ and $N$ are \emph{compatible} if $M \oplus N$ is a Dade module.
\end{definition}

There is a useful alternative way to define compatibility of Dade modules.

\begin{lemma}\label{lem:CompatibleChar}
The Dade $kG$-modules $M$ and $N$ are compatible if and only if $\Hom (M, N) \cong M^* \otimes N$ 
is  a Sylow-vertex trivial $p$-permutation module.
\end{lemma}

\begin{proof}
Dade's original argument \cite[Proposition 2.3]{Dade-Endo} applies here too. Consider the standard decomposition
\[
\End(M\oplus N)\cong \End(M)\oplus\End(N)\oplus\Hom(M,N)\oplus\Hom(N,M)
\]
of $kG$-modules, together with the isomorphism $\Hom(N,M)\cong\Hom(M,N)^\ast$. As the dual of a Sylow-vertex trivial module is Sylow-vertex trivial, 
the result is immediate.
\end{proof}
 
The following is a direct consequence of Lemma \ref{lem:CompatibleChar}. 

\begin{lemma}\label{lem:EquivRelations} 
The compatibility relation is an equivalence relation on the set of capped Dade modules. We denote this relation by $M\sim N$.
\end{lemma} 

\begin{proof} The relation is reflexive by the definition of  Dade module and Lemma \ref{lem:CompatibleChar}.
It is clearly symmetric as it is defined by a direct sum. For transitivity, assume that $M\sim N$ and $N\sim L$. Then 
$$T:=  (M^* \otimes N) \otimes (N^* \otimes L)$$ is a Sylow-vertex trivial  $p$-permutation module. Since $N$ is capped, 
we have $N \otimes N^* \cong k^n \oplus W$ for some $n \geq 1$. This gives that $M^* \otimes L$ is a summand of $T$, 
hence  the result follows from the fact that a summand of a Sylow-vertex trivial module is Sylow-vertex trivial.
\end{proof}

\begin{proposition}\label{pro:DadeGroup} 
Let $D(G)$ denote the set of equivalence classes of capped Dade $kG$-modules with respect to the compatibility relation.
Then the operation $[M]+[N]:=[M\otimes N]$ defines an abelian group structure on $D(G)$.
\end{proposition}

\begin{proof} It is clear that the tensor product of two capped Dade modules is a capped Dade module. For well-definedness 
we need to show that if $M\sim M'$ and $N\sim N'$, then $M \otimes N \sim M'\otimes N'$, which follows easily from the definition of compatibility. 
So the operation is well-defined and it is obviously commutative. The zero element is the equivalence class $[k]$ of the trivial 
module. The inverse of $[M]$ is $[M^*]$. 
\end{proof}

When we want to emphasize the field $k$, we write $D_k (G)$ for the Dade group. In the next section we show that 
the Dade group defined in Proposition \ref{pro:DadeGroup} is isomorphic to the Dade group defined by Lassueur  
(see Proposition \ref{pro:IsomDade}).

%--------------------------------

\section{Indecomposable Dade modules and Lassueur's Dade group}\label{sect:Indecomposable}
 
We begin with a lemma due to Benson and Carlson:

\begin{lemma}[{\cite[Theorem 2.1]{BensonCarlson-Nilpotent}}]\label{lem:Benson-Carlson}
Let $M$ and $N$ be indecomposable $kG$-modules.  The trivial module $k$ is a component of $M\otimes N$ 
if and only if $M\cong N^\ast$ and $p\nmid\dim_k M$, in which case $M\otimes N$ contains exactly one copy of $k$.
\end{lemma}

We will also need another result due to Lassueur. Recall  that for a fixed $kG$-module $V$, 
a $kG$-module $M$ is called \emph{$V$-projective} or \emph{projective relative to $V$}, if there is a $kG$-module $N$ 
such that $M \mid V \otimes N$.  

\begin{lemma}[{\cite[Lemma 5.1]{Lassueur-Dade}}]\label{lem:Indecomposable} 
Let $M$ be an indecomposable endo-$p$-permutation $kG$-module with vertex $S$. Then
\begin{enumerate}
\item\label{lem:Indecomsable_i} $\res^G _S M$ is a capped endo-permutation $kS$-module.
\item\label{lem:Indecomposable_ii} $p\nmid\dim_k M$.
\item\label{lem:Indecomposable_iii}$k$ is a summand of $\End(M)$ with multiplicity 1.
\end{enumerate}

\begin{proof} We give a self-contained proof here; more details can be found in \cite[Lemma 5.1]{Lassueur-Dade}.
Let $M$  be an indecomposable endo-$p$-permutation module with vertex $S$. Then $\res^G _S M$ is 
an endo-permutation module that has a summand with vertex $S$, hence it is capped (see \cite[Definition 12.2.5]{Bouc-BisetBook}). This proves \ref{lem:Indecomsable_i}. By Lemma \ref{lem:Benson-Carlson}, \ref{lem:Indecomposable_ii} and \ref{lem:Indecomposable_iii} are equivalent, so it remains to show that \ref{lem:Indecomposable_ii} holds.

Let $U$ be a component of $\res^G _S M$ with vertex $S$. By applying Higman's Criteria \cite[Prop 3.6.4]{Benson-Book1}, 
we conclude that $\End (U)\cong U^* \otimes U$ has the trivial module $k$ as a summand  (see  \cite[Lemma 12.2.6]{Bouc-BisetBook}). This gives that $k$ is projective 
relative to $U$, hence it is  also projective relative to $\res ^G_S M$.  Let $N$ be a $kS$-module such that 
$k \mid N \otimes  (\res^G _S M)$. Then $$\ind _S ^G k  \mid \ind _S ^G ( N \otimes \res^G _S M ) \cong (\ind ^G _S N )  \otimes M$$ 
by Frobenius Reciprocity. Since $[G:S]$  is coprime to $p$, we have $k \mid \ind _S ^G k \cong k[G/S]$.
Hence $k \mid (\ind ^G _S N) \otimes M$. This means that there is a component $V$ of $\ind_S^G N$ such that 
$k \mid V \otimes M$. Applying Lemma \ref{lem:Benson-Carlson} to $V$ and $M$, we conclude that $p \nmid \dim _k M$.  
This completes the proof.
\end{proof}
\end{lemma}
 
It is interesting to ask whether every indecomposable endo-$p$-permutation module is a Dade module. 
The following example shows that the answer is ``no" in general.
 
\begin{example} Let $G=D_8$ be the dihedral group of order $8$ and set $Z=Z(G)$. Let $p$ be an odd prime. Then $S=1$ and every $kG$-module is semisimple. Let $M$ be the (unique) $2$-dimensional irreducible $kG$-module. Then,
$M$ is an endo-$p$-permutation module because its restriction to $S$ is a direct sum of trivial modules. However $M$ is not a Dade module because $\End(M) \cong k[G/Z]$ has a nontrivial Sylow-vertex summand.

More generally, if $G$ is a finite group such that $S\trianglelefteq G$ with $G/S$ nonabelian, we can take $M$ to be an irreducible $G/S$-representation of degree at least 2. Viewing $M$ as a $G$-module by inflation, we have by construction that the $S$-action on $M$ is trivial, and hence $M$ is an endo-$p$-permutation $kG$-module.  We similarly have that $S$ acts trivially on $\End(M)$, so that all components of the endomorphism module have Sylow vertices.  Schur's Lemma implies that $\End(M)$ contains a unique trivial summand, but as $\dim(\End(M))>1$ there must be a second nontrivial Sylow-vertex component.  Thus $M$ is not a Dade module. 
\end{example}

We now prove the uniqueness of the cap of an endo-permutation $kS$-module by adapting Dade's original argument from \cite[Theorem 3.8]{Dade-Endo} to our situation.  

\begin{proposition}\label{pro:UniqueCaps}
Let $M$ be a capped Dade $kG$-module. If $U$ and $V$ are two Sylow-vertex components of $M$, then $U\cong V$. 
\end{proposition}

\begin{proof}  Let $U$ and $V$ be two Sylow-vertex components of $M$. Consider the $kG$-module $\Hom(U,V)\cong  U^* \otimes V$, which is a direct summand of $\End(M)$.  If we can show that $\Hom(U,V)$ contains a Sylow-vertex component, then this component will also be a component of $\End(M)$, and thus isomorphic to the trivial $kG$-module $k$ by the assumption that $\End(M)$ is Sylow-vertex trivial.  Lemma \ref{lem:Benson-Carlson} will then imply that  $U^* \cong V^*$, and hence $U \cong V $, as desired.

Consider $U\otimes U^* \otimes V \cong \End(U) \otimes V$.  Since $U$ has vertex $S$, by Lemma \ref{lem:Indecomposable} $\End(U)$ has the trivial module $k$ as a component.  Thus $V$ is a Sylow-vertex component of $U \otimes (U^* \otimes V)$.  As the vertices of the components of a tensor product of indecomposable $kG$-modules are contained in the intersection of the vertices of the modules being tensored, it follows that $U^* \otimes V $ has a Sylow-vertex component, and the result is proved.
\end{proof}

 As a consequence, every capped Dade $kG$-module $M$ has a unique (up to isomorphism) indecomposable Sylow-vertex summand. As mentioned before, we call this summand the \emph{cap} of $M$ and denote it by $Cap(M)$.  We may then define the \emph{Dade group} of $G$, written $D(G)$, to be the the set of isomorphism classes of indecomposable capped Dade $kG$-modules with abelian group structure
\[
[M]+[N]:=[Cap(M\otimes N)].
\]
Clearly the identity for this operation is $[k]$, and the inverse of $[M]$ is $[M^\ast]$. This definition gives the same Dade group as that defined in Theorem \ref{thm:DadeGroup} because of the following lemma:  

\begin{lemma}\label{lem:CapEquiv}\
\begin{enumerate}
\item\label{lem:CapEquiv_i} Let $U$ be a Sylow-vertex component of the Dade $kG$-module $M$. Then $U$ is a capped Dade module and $U \sim M$.
\item\label{lem:CapEquiv_ii} Let $U$ and $V$ be two capped indecomposable Dade $kG$-modules.  Then $U \sim V$ if and only if $U \cong V$.
\end{enumerate}
\end{lemma}

\begin{proof}\
\begin{enumerate}
\item Since $U^*  \otimes U$ is a summand of $M^* \otimes M$, we conclude that $U$ is a capped Dade module.  Note that $U^* \otimes M$ is a summand of $M^* \otimes M$, hence $U^* \otimes M$ is a Sylow-vertex trivial $p$-permutation module. Hence by Lemma \ref{lem:CompatibleChar}, $U$ is compatible with $M$.
\item If $U$ and $V$ are compatible, then $U\oplus V$ is a Dade $kG$-module.  Applying Proposition \ref{pro:UniqueCaps}, we obtain that $U \cong V$.
For the converse, observe that if $U \cong V$, then $U^* \otimes V \cong U^* \otimes U$ is a Sylow-vertex trivial $p$-permutation module. Hence $U \sim V$ by Lemma \ref{lem:CompatibleChar}.
\end{enumerate}
\end{proof}

The isomorphism between these two definitions of the Dade group is given by 
\[
[M] \mapsto [Cap(M)], 
\]
where the brackets on the left indicate the equivalence class of the capped Dade module, while those  
on the right refer to the isomorphism class of the capped indecomposable Dade module.
 
We now discuss Lassueur's definition of the Dade  group of a finite group, where the subject 
is approached using relative endotrivial modules (see \cite[Cor.Def. 5.5]{Lassueur-Dade}). 
If $V$ is a $kG$-module such that the trivial module $k$ is  $V$-projective, 
then all $kG$-modules are $V$-projective. It is therefore most interesting to consider the case where $k$ is not $V$-projective. 
If $V$ has a summand $V_i$ such that $p \nmid \dim _k V_i$, then $k \mid \End (V_i) $, and hence $k \mid \End (V)\cong V^* \otimes V$.
Therefore if we want $k$ not to be $V$-projective, we must have $p\mid \dim _k V_i$ for every component $V_i$ of $V$ (see \cite[Prop 2.6]{Lassueur-Dade}). A $kG$-module $V$ satisfying this property is called \emph{absolutely $p$-divisible}. 

\begin{definition} Let $V$ be an absolutely $p$-divisible $kG$-module.
A $kG$-module $M$ is \emph{endotrivial relative to $V$}, or simply \emph{$V$-endotrivial}, if $\End(M)$ 
is isomorphic to $k \oplus W$ for some $V$-projective $kG$-module $W$. In this case we write 
$\End(M)\cong k \oplus (V-proj)$.
\end{definition}

Recall that a $kG$-module is called \emph{endotrivial} if $\End (M) \cong k\oplus W$ for $W$ a projective $kG$-module. Hence an endotrivial module is a $V$-endotrivial module in the case where $V=kG$. Relative endotrivial modules 
have properties similar to endotrivial modules (see \cite[Lemma 2.8]{Lassueur-Dade}). In particular, given a 
$V$-endotrivial module $M$, there is a direct sum decomposition $M \cong M_0 \oplus (V-proj)$, where 
$M_0$ is the unique indecomposable summand of $M$ that is $V$-endotrivial. This summand is called the \emph{cap} of $M$.

Two $V$-endotrivial modules $M$ and $N$ are declared equivalent if their caps $M_0$ and $N_0$ are isomorphic. 
Lassueur defines the group $T_V(G)$ of $V$-endotrivial modules as the group of equivalence classes of 
$V$-endotrivial modules, with addition given by $[M]+[N] :=[M \otimes N]$.

To define the Dade group of a finite group, Lassueur considers $V$-endotrivial modules for a special choice of permutation module  $V$.

\begin{definition}\label{def:F_G}  
Let $\cF_G$ denote the set of all non-Sylow $p$-subgroups of $G$,
and $\cF_G/G$ the $G$-conjugacy classes of subgroups in $\cF_G$. The $kG$-module $V(\cF_G)$ is defined to be the permutation module  $$V(\cF_G ):= \bigoplus _{[Q]\in \cF_G/G} k[G/Q].$$
\end{definition}

It is easy to see that $k$ is not $V(\cF_G)$-projective since $k$ has vertex $S$. In particular, $V(\cF_G)$ is absolutely $p$-divisible.
Note that a module's being $V(\cF_G)$-projective is equivalent to its not having any Sylow-vertex components. For $V (\cF_G)$-endotrivial modules Lassueur proves the following result:

\begin{proposition}[\cite{Lassueur-Dade}, Proposition 5.2]\label{pro:StronglyCapped}  
Let $M$ be an endo-$p$-permutation module. The following are equivalent:
\begin{enumerate}
\item\label{pro:StronglyCapped_i} $M$ is $V(\cF_G)$-endotrivial.
\item\label{pro:StronglyCapped_ii} $\res^G_S M$ is $V (\cF_S)$-endotrivial.
\item\label{pro:StronglyCapped_iii} $M$ has a unique indecomposable summand $M_0$ with vertex $S$. 
\item\label{pro:StronglyCapped_iv} $\End(M) \cong k\oplus W$, where $W$ is a $p$-permutation $kG$-module 
all of whose components have non-Sylow vertices.   
\end{enumerate}
\end{proposition}

Lassueur calls an endo-$p$-permutation module \emph{strongly capped} if it satisfies the equivalent conditions 
of Proposition \ref{pro:StronglyCapped}. The unique indecomposable summand $M_0$ of $M$ is called the \emph{cap} of $M$. 
Similar to the $V$-endotrivial case, one defines $M \sim N$ if $M_0 \cong N_0$. Lassueur defines the Dade group 
$D(G)$ of a finite group $G$ as the group of equivalence classes of strongly capped endo-$p$-permutation modules 
with addition given by $[M]+[N]:=[M\otimes N]$. 

Lassueur's definitions are related to our notion of a Dade module by the following observation. 
 
\begin{lemma}\label{lem:StronglyCappedDade}
The endo-$p$-permutation $kG$-module $M$ is strongly capped if and only if it is a capped Dade module with a unique copy of its cap.
\end{lemma}

\begin{proof} If $M$ is strongly capped,  then by Lemma \ref{pro:StronglyCapped}\ref{pro:StronglyCapped_iv}
$M$ is a Dade module such that $\End(M)$ has a single copy of $k$. Note that if a Dade module 
has more than one copy of its cap, then $\End(M)$ has more than one copy of $k$ 
as a direct summand. This shows that $M$ must have only one copy of its cap.

For the other direction, observe that if $M$ is an indecomposable Dade module with vertex $S$ then
by Lemma \ref{lem:Indecomposable} the trivial module $k$ has multiplicity one in $\End(M)$. Hence $M$ is strongly capped.
This shows that if $M$ is a Dade module with a unique copy of its cap, then $M$ is strongly capped.
\end{proof}

As a consequence of Lemma \ref{lem:StronglyCappedDade} we can conclude the following:

\begin{proposition}\label{pro:IsomDade} The Dade group $D(G)$ of a finite group $G$ defined in 
Theorem \ref{thm:DadeGroup} is isomorphic to the Dade group defined by Lassueur in \cite[Cor.Def. 5.5]{Lassueur-Dade}.
\end{proposition}

\begin{proof} This follows from Lemma \ref{lem:CapEquiv} and the discussion thereafter. If we write $D'(G)$ for Lassueur's Dade group, 
then there is a natural map $D'(G) \to D(G)$ that takes the equivalence class $[M]$ of a strongly capped endo-$p$-permutation module $M$ to its equivalence class $[M]$ in $D(G)$. In the other direction, we have $D(G) \to D'(G)$ defined by $[M] \mapsto [Cap(M)]$. It is easy to see by Lemma \ref{lem:CapEquiv} that these maps are well-defined mutual inverses. 
\end{proof} 

\begin{remark} In \cite{Lassueur-Thesis}, Lassueur also considers the possibility of defining the Dade group using endo-$p$-permutation 
modules that contain multiple copies of their caps. Such endo-$p$-permutation modules are termed \emph{weakly capped}, however an independent condition for being weakly capped  is not given. By Proposition \ref{pro:UniqueCaps}, a Dade module is a weakly capped $kG$-module, so our definition of the Dade group $D(G)$ coincides with $\widetilde D(G)$ defined on \cite[pg. 92]{Lassueur-Thesis}.
\end{remark}

%---------------------

\section{Tensor induction of Dade modules}\label{sect:TensorInd}

As we have shown in the previous section, the notion of a Dade module is a slight generalization of Lassueur's ``strongly capped'' endo-$p$-permutation modules:  A capped Dade module may contain multiple (isomorphic) caps, while a strongly capped endo-$p$-permutation  module by definition has a unique Sylow-vertex component.  Our relaxation  was motivated by the fact that the class of strongly capped endo-$p$-permutation modules is not closed under tensor induction (see \cite[Counterexample 6.3]{Lassueur-Dade}). We show below that in the example considered by Lassueur, tensor 
induction of a particular strongly capped endo-$p$-permutation module does yield a Dade module, even though it is no longer strongly capped. However, 
we also give a different example to show that the class of Dade modules is not generally closed under tensor induction.
 
Recall that for a subgroup $H\leq G$, the tensor-induced $kG$-module $\ten _H ^G M$ of the $kH$-module $M$ is defined 
analogously to the usual induction $\ind _H ^G M$, but with the direct sum replaced by a tensor product:
\[
\ind_H^GM:=\bigoplus_{gH\in G/H}g\cdot M\qquad\textrm{vs.}\qquad
\ten_H^GM:=\bigotimes_{gH\in G/H}g\cdot M.
\]
The $G$-action  on  $\ten_H^G M$ is induced by the embedding $G\hookrightarrow \Sigma_{G/H}\wr H$.
We refer the reader to \cite[Section 3.15]{Benson-Book1} for more details.

Let $G$ be a finite group with $S\in\Syl_p(G)$, and let $\Omega$ denote the kernel of the augmentation map $\varepsilon: kS\to k$. By Alperin's theorem on relative syzygies \cite{Alperin-Construction},  the $kS$-module
$\Omega$ is an endo-permutation module. An easy calculation shows that 
\[
\End (\Omega ) \cong \Omega ^*  \otimes \Omega  \cong k\oplus m \cdot k[S/1]
\]
as $kS$-modules, where $m=|S|-2$. Since tensor induction commutes with tensor products we have 
\[
\End (\ten _S ^G \Omega) \cong \ten _S ^G ( \End (\Omega )) \cong \ten _S ^G (k \oplus m\cdot k[S/1]).
\]
Let $X:=[S/S]+ m\cdot [S/1]$. We want to compute $\ten_S^G (kX)$. 

For an $S$-set $X$, let $\jnd_S^G(X)$ denote the $G$-set $\Map_S(G,X)$, where the $G$-action is given by the rule $g\cdot\varphi:g'\mapsto\varphi(g'g)$. This determines a functor $\jnd_S ^G :S\text{-set}\to G\text{-set}$, called \emph{multiplicative
induction} (see \cite[Sect. 3]{Yalcin-Induction}). From the definition of tensor induction, and by the fact that $kX\otimes kY\cong k[X\times Y]$, it is easy to see that
\[
\ten _S^G (kX) \cong k[\jnd _S ^G X]
\]
(see \cite[Example 12.4.10]{Bouc-BisetBook}). As $\End(\ten_S^G \Omega)\cong\ten_S^G (kX)\cong k[\jnd_S^GX]$ is a permutation $kG$-module, $\ten_S^G(\Omega)$ is endo-$p$-permutation.  To see that $\ten_S^G(\Omega)$ is a Dade $kG$-module, it remains to show that the Sylow-vertex components of its endomorphism module are trivial.

We can write $$\jnd_S^G X=\sum\limits_{[H\leq G]} n_H\cdot [G/H]$$ for some natural numbers $n_H \geq 0$. This gives the $kG$-module isomorphism
\[
\End(\ten_S^G \Omega)\cong k[\jnd_S^G X]\cong\bigoplus_{[H\leq G]} n_H\cdot k[G/H].
\]
Note that if $T$ is  a Sylow $p$-subgroup of $H$, then the summand $k[G/H]$ is relatively $T$-projective.  From this we can conclude that the only components of $k[\jnd_S^G(X)]$ that could possibly be Sylow-vertex are summands of $k[G/H]$ where $p\nmid[G:H]$. This observation allows us to prove the following:
 
\begin{lemma}\label{lem:Normal} With  notation as above, if $O_p(G)\neq 1$ then $\ten_S^G \Omega$ is a strongly capped Dade module.
\end{lemma}

\begin{proof}  For a $G$-set $Y$ and a subgroup $K \leq G$, write $\f _K (Y) $ for the order of the fixed-point set $Y^K$.  Recall that $Y$ is determined as a $G$-set by the values $\{ \f _K (Y)\}$ as $K$ ranges over the $G$-conjugacy classes of subgroups of $G$. For multiplicative induction we have the following formula  (see \cite[Sect. 3]{Yalcin-Induction}):
\begin{equation}\label{eqn:FixedPoints}
\f _K (\jnd _S ^G X) = \prod _{Kg S \in K\backslash G /S } \f _{K {}^g\cap S } (X).
\end{equation}
Since $X=[S/S]+m\cdot[S/1]$, we have $\f _{K^g\cap S} (X)=1$ whenever $K ^g\cap S \neq 1$. If $K$ is a subgroup of $G$ such that 
$p \nmid [G:K]$, then $K$ must contain a Sylow $p$-subgroup of $G$, and hence $K$ contains $O_p(G)\neq 1$. This implies that 
for every $g \in G$ we have $K^g\cap S \neq 1$, so $\f _K(\jnd_S^G(X))=1$ for every $K$ with $p'$-index. 
It follows that $n_G=1$ and $n_K=0$ for all  $K\lneq G$ with $p'$-index.  Thus 
$\End(\ten_S^G \Omega)\cong k[\jnd_S^G(X)]$ has a unique Sylow-vertex component, which is $k$.  
The result is proved.
\end{proof}

When $O_p(G)=1$, the conclusion of Lemma \ref{lem:Normal} need not hold, as Lassueur illustrates in  \cite[Counterexample 6.3]{Lassueur-Dade}.

\begin{example} Consider the group $G := C_7 \rtimes C_3$ at the prime $p=3$. In this case $S=C_3$ and $\Omega$ is the unique 2-dimensional irreducible $kS$-module. By direct computation using the fixed-point formula of Equation \ref{eqn:FixedPoints}, we can see that 
\[
\End (\ten _S ^G \Omega ) \cong k[G/G] \oplus 3 \cdot k[G/C_7] \oplus 15 \cdot k[G/C_3] \oplus 7\cdot k[G/1].
\]
Note that, as $\res_S^G([G/S])=[S/S]+2\cdot[S/1]$, we have $k[G/S]\cong k\oplus U$, where no component of $U$ is Sylow-vertex. Lassueur concludes that in this case $\ten_S^G \Omega$ is not strongly capped. However, it is also clear from the computation that $\ten_S^ G \Omega$ is a Dade module. 
\end{example}

Unfortunately, the tensor induction $\ten_H ^G M$   of a Dade module $M$ need not be a Dade module in general (even when $H$ is of coprime index), as the following example shows:

\begin{example}\label{ex:TensorCounterExample} Let $G=\Sigma_4$ be the symmetric group on 4 letters, and  take $p=3$. Then $S\in\Syl_p(G)$ is again $C_3$, the module $\Omega$ is a 2-dimensional $kS$-module, and $X\cong[S/S]+[S/1]$.
In this case we have 
$$\End (\ten _S ^G (\Omega )) \cong k[G/G] \oplus 3\cdot  k[G/S_3] \oplus 6\cdot  k[G/S] \oplus  \cdots$$
The coefficients above were computed using Equation \ref{eqn:FixedPoints} and the fixed-point values $\f _G(X)=1$, $\f _{A_4}(X)=1$, $\f _{S_3}(X)=4$, and $\f _{S}(X) =16$. Note that 
\[
\f_{A_4} ( \jnd _{S} ^G (X)= \prod _{g\in A_4\backslash S_4 /S} \f _{A_4^g \cap S} (X) = (\f _{S} (X)) ^2= 1
\]
since $\f _{S} (X)=1$. We also have 
\[
\f_{S} ( \jnd _{S} ^G (X) )= \prod _{g\in S\backslash S_4 /S} \f _{S^g\cap S} (X)= (\f _{S} (X))^2 (\f _1(X))^2=16
\]
 since $\f _{S} (X)=1$ and $\f _1(X)=4$. The values of $\f _K$ for other subgroups are computed in a similar way.

The module $k[G/S]$ contains a copy of the alternating representation.  It follows that $\End(\ten_S^G(\Omega))$ has a non-trivial Sylow-vertex component, so $\ten_S^G(\Omega)$ is not a Dade module. \qed
\end{example}

On the positive side, for $G=\Sigma _3$, we do have  $\ten _{C_3} ^{G} \Omega$ is a Dade module by Lemma \ref{lem:Normal}. In this case we have $$\End (\ten _{C_3} ^{G} \Omega )=\ten _{C_3} ^{G} (kX) = k \oplus 3k[G/C_2] \oplus k[G/1].$$ 
This is a special case of a more general observation that holds for finite groups with normal Sylow $p$-subgroup.

\begin{lemma}\label{lem:Restriction} Let $G$ be a finite group such that $S \trianglelefteq G$. Then for every indecomposable endo-permutation $kS$-module $M$ with vertex $S$, the tensor-induced module $\ten_S^G M$ is a capped Dade module.
\end{lemma}

\begin{proof} Let $X$ be the $S$-set such that $\End (M)\cong kX$. By the assumption that $M$ has vertex $S$, we have $$X=[S/S]+\sum _{[P<S]} n_P \cdot [S/ P],$$ where the sum is taken over the $S$-conjugacy classes of proper subgroups of $S$. 
 (The existence of the summand $k=k[S/S]$ follows from \cite[Lemma 12.2.6]{Bouc-BisetBook}, its uniqueness from Lemma \ref{lem:Benson-Carlson}). As above we have 
$$\End (\ten _S ^G M) \cong \ten _S ^G (kX)\cong k[\jnd _S ^G X],$$ so $\ten _S^G M$ is an endo-$p$-permutation module.
To find the components of $\End( \ten_S^G M)$ with vertex $S$, we need to consider the orbits $[G/K]$ of $\jnd _S^G X$ 
for those $K$ satisfying  $K \geq S$. Since $S \trianglelefteq G$, we have 
$$ \f_S (\jnd _S ^G X)=\prod _{gS \in G/S} \f_{S} (X) =1.$$ 
Since $f_G( \jnd _S ^G X) =1$, we have $[G/G]$ is an orbit in $\jnd _S^G (X)$. Since $f_S(\jnd _S ^G X) =1$,  the $G$-set
$\jnd _S ^G X$ has no orbits $[G/K]$ with $S\leq K <G$. Hence $\ten_S^G M$ is a capped Dade module.
\end{proof}

In certain circumstances, there is a restriction homomorphism between Dade groups.  
 
\begin{lemma}\label{lem:Restrict}  Let $S$ be a Sylow $p$-subgroup of $G$, and let $H$ be a subgroup of $G$ containing $S$. 
Then the restriction of a capped Dade $kG$-module to $H$ is a capped Dade $kH$-module. 
This induces a well-defined group homomorphism $\res^G _H : D(G) \to D(H)$.
\end{lemma}

\begin{proof} For every $kG$-module $M$, we have $$\End (\res ^G _H M) \cong \res^G _H (\End (M)) \cong \res ^G _H (k^n \oplus W)$$
for some endo-$p$-permutation $kG$-module $W$ whose components have vertices strictly smaller than $S$.  Since  $S$ is also a Sylow $p$-subgroup of $H$, all the components of $\res^G _H W$ have non-Sylow vertices, so $\res^G _H M$ is a capped Dade $kH$-module.
\end{proof} 
  
The restriction homomorphism can be also described using strongly capped endo-$p$-permutation modules (see  \cite[Lemma 6.1]{Lassueur-Dade}).
In general the restriction map $\res^G _H$  is not injective, but it is when $H$ contains $N_G(S)$. 

\begin{proposition}[Lassueur \cite{Lassueur-Dade}, Lem. 6.1] 
Let $H\leq G$ be a subgroup such that $N_G(S) \leq H$. Then the restriction map $\res^G _H : D(G) \to D(H)$ is an injection.
\end{proposition}

\begin{proof} Suppose that $\res^G _H [M_1] = \res ^G _H [M_2]$. Then $\res ^G _H (M_1\oplus M_2) $ is a Dade $kH$-module. By Lemma \ref{lem:Restrict}, $\res^G _{N_G(S) } (M_1\oplus M_2) $ is also a Dade module. Applying Lemma \ref{lem:Detection} we obtain that $M_1 \oplus M_2$ is a Dade $kG$-module. We conclude that $[M_1]=[M_2]$.
\end{proof}

The restriction map $\res^G _H : D(G) \to D(H)$ is not defined for an arbitrary subgroup $H \leq G$.  Even though the restriction of an endo-$p$-permutation module is always an endo-$p$-permutation module, the restriction $\res^G _H M$ of a Dade $kG$-module $M$ may not be a Dade $kH$-module, due to the Sylow-vertex triviality condition.
 
\begin{example}\label{ex:Restriction}  Consider the group $G=\Sigma _3$ at the prime $p=3$. Let $H=C_2$ and $\Delta (G/H)$ denote the kernel of the augmentation map $\varepsilon: k[G/H] \to k$. By direct computation, one can see that the relative syzygy $\Delta (G/H)$ is a capped Dade module (or apply Proposition \ref{pro:RelativeDade}). Note that $\Delta (G/H)$ is an indecomposable $kG$-module because its restriction to $S$ is indecomposable. Set $N:=\res ^G _H  (\Delta (G/H)) \cong k[H/1] \cong k \oplus k(-1)$, where $k(-1) $ denotes the one-dimensional sign representation for $H$.  
Then $N$ is not a Dade $kH$-module because $\End (N)\cong \End(k \oplus k(-1) )$ has two Sylow-vertex components isomorphic to $k (-1)$. 
\end{example}
   
We end this section with a result by Lassueur on the structure of the Dade group $D(G)$. Following Lassueur's notation, let $X(N_G(P))$ denote the  group of one-dimensional representations of $N_G(S)$.  We write $\Upsilon (G)$ for the subgroup of $D(G)$ whose elements are the isomorphism classes of the $kG$-modules arising as the Green correspondents of modules in $X(N_G(S))$. It is shown in \cite[Proposition 4.1]{Lassueur-Dade} that such modules are indeed strongly capped endo-$p$-permutation modules.
 
\begin{theorem}[Thm. 7.3, \cite{Lassueur-Dade}]\label{thm:LassueurExact} Let $G$ be a finite group with Sylow $p$-subgroup $S$. Then there is an exact sequence of abelian groups 
$$ 0 \to \Upsilon (G) \maprt{} D(G) \maprt{\res ^G _S} D(S) ^{G\text{-st} } \to 0,$$
where $D(S)^{G\text{-st}}$ denotes the subgroup formed by $G$-stable elements in $D(S)$ as defined in \cite[Def. 1.4]{Urfer-Endo}.
\end{theorem}
 
This is an important structural result for the Dade group of a finite group that makes computation possible in many cases. In particular,
Theorem \ref{thm:LassueurExact}  shows that $D(G)$ is a finitely generated abelian group, as both $\Upsilon (G)$ and $D(S)^{G\text{-st}}$ are finitely generated.

%----------------------------------------
 
 \section{The Dade group generated by relative syzygies}\label{sect:RelativeSyzygies}

In this section, we describe an important class of Dade modules that will be the focus of the remainder of the paper.
%In \cite{Alperin-Construction}, Alperin gave a primary source of examples of endo-permutation modules for a $p$-groups.
For $G$ an arbitrary finite group and $X$ a non-empty finite $G$-set,  let $\Delta(X)$ denote the kernel of the augmentation map $\varepsilon:k X\to k$ defined by  
$$\sum\limits_{x\in X} c_x \cdot x\mapsto \sum\limits_{x\in X}c_x.$$  The $kG$-module  $\Delta(X)$ is called the \emph{relative syzygy 
of $X$}. Throughout the paper we assume all $G$-sets are non-empty and finite. 

When $G=S$ is a $p$-group, Alperin showed in \cite[Theorem 1]{Alperin-Construction} that $\Delta(X)$ is always an 
endo-permutation $kS$-module and it is capped if $|X^S| \neq 1$. For an $S$-set $X$ with $|X^S| \neq 1$, we define the element $\Omega_X:=[\Delta(X)]\in D(S)$.  When $|X^S|=1$, we declare that $\Omega_X=0$. The set of elements $\{\Omega_X\}$, as $X$ runs over all finite $S$-sets, generates a subgroup $D^{\Omega} (S)$ of the Dade group $D(S)$, called \emph{the Dade group generated by relative syzygies} or  \emph{the group of relative syzygies}. This group plays an important role in the calculation of the full Dade group for $p$-groups. 

Our goal in this section is to show that Alperin's construction works in essentially the same way to give rise to elements of $D(G)$ when $G$ is an arbitrary finite group.  We follow the arguments of \cite{Bouc-Tensor}, which actually do the vast majority of the work for us.  Some of the results in this section also appear in \cite{Lassueur-Relative} and  \cite{Lassueur-Dade} in different forms, using the language of relative endotrivial modules. We begin by recalling the definition of $X$-split sequences.

\begin{definition}(\cite[Def. 2.1.1]{Bouc-Tensor}) \label{def:X-split} 
Let $X$ be a finite $G$-set. An exact sequence of $kG$-modules $0\to L\to M\to N\to 0$ is called \emph{$X$-split} if it splits after tensoring with $k X$, i.e., $0\to L\otimes k X\to M\otimes k X\to N\otimes k X\to 0$ is a split-exact sequence of $kG$-modules.   
\end{definition}

There is also a notion of relative $X$-projectivity defined using a list of equivalent conditions.

\begin{definition}(\cite[Def. 2.2.1]{Bouc-Tensor}) \label{def:X_proj}
Let $X$ be a finite $G$-set. A $kG$-module $M$ is called \emph{relatively $X$-projective} if the following equivalent conditions are satisfied:
\begin{enumerate}
\item \label{def:X_proj_ii_a} $M$ is a direct summand of $k X\otimes N$ for some $kG$-module $N$.
\item\label{def:X_proj_ii_c}  The map $k X\otimes M\to M$ defined by $x\otimes m\mapsto m$ is a split epimorphism of $kG$-modules.
\item\label{def:X_proj_ii_d} The lifting problem
\[
\xymatrix{
&M\ar[d]^-f\ar@{..>}[dl]_-{\widetilde f}\\
L\ar[r]_\varphi&N\ar[r]&0
}
\]
has a solution whenever $\varphi\otimes\id_{k X}:L\otimes k X\to  N\otimes k X$ is a split epimorphism of $kG$-modules.
\end{enumerate}
\end{definition}

These conditions are indeed equivalent, as can be seen by using arguments similar to the standard ones given for projective $kG$-modules (see \cite[Sect. 3.6]{Benson-Book1}).   
 
We now state a relative version of Schanuel's Lemma  (\cite[Proposition 2.3.1]{Bouc-Tensor}): 

\begin{lemma}[Relative Schanuel's Lemma] \label{lem:Schanuel}
For a finite $G$-set $X$, suppose that
\[
\xymatrix@=1.5em{
0\ar[r]&L\ar[r]&M\ar[r]&N\ar[r]\ar@{=}[d]&0\\
0\ar[r]&L'\ar[r]&M'\ar[r]&N\ar[r]&0
}
\]
are two $X$-split short exact sequences of $kG$-modules, where $M$ and $M'$ are both relatively $X$-projective.  Then $M\oplus L'\cong M'\oplus L$.
\end{lemma}

\begin{proof} The proof is similar to the proof of the well-known version of Schanuel's lemma. See \cite[Lemmas 3.9.1 and 1.5.3]{Benson-Book1}.
\end{proof}
 
We have the following technical result.

\begin{lemma}\label{lem:Decomp}
Let $X$ be a finite $G$-set and let $\Delta (X)$ denote the kernel of the augmentation map $\varepsilon: kX \to k$.
Then there is an isomorphism of $kG$-modules
\[
\End(\Delta(X))\oplus  k X \oplus kX \cong k \oplus  k[X\times X].
\]
\end{lemma} 

\begin{proof}  In the case that $G=S$ is a $p$-group, this result is due to Alperin \cite[Theorem 1]{Alperin-Construction}.  His proof also works for finite groups. Here we give a proof based on an argument due to Bouc \cite[Lemma 2.3.3]{Bouc-Tensor} that uses the Relative Schanuel's Lemma.

We begin with the  exact sequence
\[
0\to\Delta(X)\to k X\to k\to 0,
\]
which we claim to be $X$-split.  Indeed, after tensoring with $kX$ we have
\[
0\to\Delta(X)\otimes k X\to k X\otimes k X\to k X\to 0,
\]
where the last map is given by $x'\otimes x\mapsto x$.  Clearly $k X\to k X\otimes k X:x\mapsto x\otimes x$ defines a splitting, as desired.   
It is easy to see that the dual sequence
\[
0\to k\to k X\to\Delta(X)^\ast\to 0
\]
is $X$-split as well, using $(k X)^\ast\cong k X$.  Note that this implies that
\[
k X\otimes k X\cong(\Delta(X)^\ast\otimes k X)\oplus k X.
\]

On the other hand, starting from our original exact sequence, if we tensor with $\Delta(X)^\ast$ we obtain
\[
0\to\End(\Delta(X))\to k X\otimes\Delta(X)^\ast\to\Delta(X)^\ast\to 0.
\]
As an $X$-split exact sequence clearly remains $X$-split after tensoring with a $kG$-module, this second sequence is $X$-split.  Finally, noting that both $k X$ and $k X\otimes\Delta(X)^\ast$ are relatively $X$-projective by definition, the Relative Schanuel's Lemma gives
\[
\End(\Delta(X))\oplus k X\cong k\oplus(k X\otimes\Delta(X)^\ast).
\]
Adding $k X$ to both sides, the isomorphism $k X\otimes k X\cong(k X\otimes\Delta(X)^\ast)\oplus k X$ noted above yields
\[
\End(\Delta(X))\oplus k X\oplus k X\cong k\oplus(k X\otimes k X),
\]
from which the desired isomorphism is obtained using $k X\otimes kX \cong k[X \times X]$.  
\end{proof}

We can now prove the first main result of this section.

\begin{proposition}\label{pro:RelativeDade}
Let $X$ be a finite $G$-set such that $X^S=\emptyset$.  Then $\Delta(X)$ is a capped Dade $kG$-module.
\end{proposition}

\begin{proof} By Lemma \ref{lem:Decomp}, there is an isomorphism of $kG$-modules
\[
\End(\Delta(X))\oplus  k X \oplus kX \cong k \oplus  k[X\times X].
\]
This shows that $\End(\Delta(X))$ is a summand of a permutation module, and hence it is an $p$-permutation module by Definition  \ref{def:pPermutation}.
Note that since $X^S=\emptyset$,  the product $X\times X$ has no $S$-fixed-points.
This gives that the right hand side of the above isomorphism has a unique Sylow-vertex component, namely the trivial module $k$. Since $k X$ has no Sylow-vertex components, it follows that $k$ must be a summand of $\End(\Delta(X))$ and that $\End(\Delta(X))$ has no other Sylow-vertex components.  As this is precisely the definition of a capped Dade module, the result is proved.
\end{proof}

The condition that $X^S =\emptyset$  is essential for the conclusion of Proposition \ref{pro:RelativeDade}.

\begin{example}
Suppose that $S \lhd G$, so that $G/S$ is a nontrivial $p'$-group.  Assume moreover that $G/S$ is abelian.  As before, let $X(G)$ denote the group of 1-dimensional characters $\chi: G \to k^\times$. It is easy to see that
\[
k[G/S]\cong\bigoplus_{\chi \in X(G)}\chi,\qquad\textrm{so}\qquad
\Delta(G/S)\cong\bigoplus_{1\neq\chi\in X(G)}\chi,
\]
where $1$ denotes the trivial character.  It follows that
\[
\End(\Delta(G/S))\cong\bigoplus_{\substack{\chi,\psi\in X(G)\\\chi\neq1\neq\psi}}\chi^\ast\otimes\psi.
\]
In particular, so long as $|G/S|>2$, at least one of these terms is a 1-dimensional (hence Sylow-vertex) nontrivial $kG$-module, so $\Delta(G/S)$ is not a Dade module.
\end{example}
 
\begin{definition}
If $X$ is a finite $G$-set such that $X^S= \emptyset$, then we write $\Omega_X$ for the class $[\Delta(X)]\in D(G)$. 
The \emph{group of relative syzygies} of $G$ is the subgroup $D^\Omega(G)\leq D(G)$ generated by the set 
$\{ \Omega_X \}$ as $X$ runs over all $G$-sets satisfying $X^S=\emptyset$. 
\end{definition}
 
 As noted above, allowing only the relative syzygies $\Delta (X)$ with  $X^S =\emptyset$ as generators for $D^{\Omega} (G)$ is a necessary condition, however this makes it difficult to capture the elements in $\Upsilon (G)$ (cf. Theorem \ref{thm:LassueurExact}) as elements in $D^{\Omega} (G)$. For example, when the order of $G$ is coprime to $p$, then by definition $D^{\Omega} (G)=0$ even though there may be many 
one-dimensional representations giving nonzero elements of $\Upsilon (G)$.  However it is still possible to produce some elements in $\Upsilon (G)$ using relative syzygies with the condition $X^S =\emptyset$, as shown in the following example, so in general $\Upsilon (G) \cap D^{\Omega } (G) \neq 0$.

\begin{example}\label{ex:S3Example}  
Consider the group $G=\Sigma_3$ at the prime $p=3$.  We have $G=\langle\sigma:=(123),\tau:=(12)\rangle$, so that any $kG$-module is determined by a pair of compatible actions of $\sigma$ and $\tau$.  Write $S:=\langle\sigma\rangle\cong C_3$ and $H:=\langle\tau\rangle\cong C_2$.

We consider the $kG$-module $M:=\Delta (G/H)=\ker(k[G/H]\xrightarrow\varepsilon k)$.   
With respect to the basis $\{x:=\sigma H-H,y:=\sigma^2 H-H\}$ for $\Delta (G/H)$, the generators of $G$ act as
\[
\sigma\rightsquigarrow
\begin{bmatrix}[r]
-1&-1\\
1&0
\end{bmatrix}
\qquad\textrm{and}\qquad
\tau\rightsquigarrow
\begin{bmatrix}[r]
0&1\\
1&0
\end{bmatrix}.
\]
 We are interested in computing the element $2[M]=[M\otimes M]$ in $D(G)$. With respect to the basis
\[
\{\alpha:=x\otimes x,\ \beta:=x\otimes y,\ \gamma:=y\otimes x,\ \delta:=y\otimes y\},
\]
the action of $G$ on $M\otimes M$ is given by the matrices
\[
\sigma\rightsquigarrow
\begin{bmatrix}[r]
1&1&1&1\\
-1&0&-1&0\\
-1&-1&0&0\\
1&0&0&0
\end{bmatrix}
\qquad\textrm{and}\qquad
\tau\rightsquigarrow
\begin{bmatrix}[r]
0&0&0&1\\
0&0&1&0\\
0&1&0&0\\
1&0&0&0
\end{bmatrix}.
\]
If we now change our basis to the ordered basis $\{ v_1, v_2, v_3, v_4\}$, where $$ v_1:=\beta-\gamma, \ v_2:=\alpha-\beta-\gamma+\delta, \ v_3:=\delta, 
\textrm{ and }\ v_4:=\alpha,$$ then the matrices with respect to this new basis are
\[
\sigma\rightsquigarrow
\begin{bmatrix}[r]
1&&&\\
&0&0&1\\
&1&0&0\\
&0&1&0
\end{bmatrix}
\qquad\textrm{and}\qquad
\tau\rightsquigarrow
\begin{bmatrix}[r]
-1&&&\\
&1&0&0\\
&0&0&1\\
&0&1&0
\end{bmatrix},
\]
hence we conclude that $M\otimes M\cong\langle v_1\rangle\oplus\langle v_2,v_3,v_3\rangle$.  The one-dimensional summand  $\langle v_1 \rangle $ is 
the sign representation $k(-1)$ of $G=\Sigma_3$, and the 3-dimensional summand is isomorphic to the permutation module $k[G/H]$.  It follows that $2\cdot [M]=[k(-1)]\in D(G)$, or that $[M]$ has order 4 in the Dade group of $\Sigma_3$. The computation $2\cdot [M]=k(-1)$ shows that the group 
$\Upsilon ^{\Omega} (G) := \Upsilon (G) \cap D^{\Omega } (G)$ is not trivial when $G=\Sigma_3$.
\end{example}

By Theorem \ref{thm:LassueurExact}, the Dade group $D(G)$ is finitely generated, hence $D^{\Omega} (G)$ is also finitely generated. 
As there are infinitely many $G$-sets $X$ satisfying $X^S=\emptyset$, there must be many relations among the relative syzygies $\{ \Omega_X \}$.
In particular the assignment $X \to \Omega_X$ is not a one-to-one correspondence. 

In the rest of the section we prove a sequence of lemmas, each of which is designed to impose relations on the generating set $\{ \Omega_X\}$. At the end of the section we use these relations to prove that $D^{\Omega} (G)$ has an explicit finite generating set, namely those $\Omega_{G/P}$ where $P$ is a non-Sylow $p$-subgroup of $G$ (Proposition \ref{prop:Generates}).

\begin{lemma}\label{lem:FixedPoints}
Let $X$ and $Y$ be $G$-sets such that $X^S=Y^S=\emptyset$.  If for every $p$-subgroup $P \leq G$, the fixed-point set $X^P\neq \emptyset$ 
precisely when $Y^P \neq \emptyset$, then  $\Omega_X=\Omega_Y$ in $D(G)$.
\end{lemma}

\begin{proof} The argument given in \cite[Lem. 3.2.7]{Bouc-Tensor} also applies here. The condition implies that there exist $S$-maps $f : X \to Y$ and $f': Y \to X$. Note that a sequence of $kG$-modules splits if and only if its restriction splits as a sequence of $kS$-modules, since a splitting of $kS$-modules can be averaged to get a splitting of $kG$-modules. This implies that a sequence of $kG$-modules is $kX$-split if and only if its restriction of $S$ is $k[\res^G _S X]$-split.

Consider the sequence  
\begin{equation}\label{eqn:Y-sequence}
0 \to \Delta (Y) \to kY \to k \to 0.
\end{equation}
This sequence is $Y$-split, hence its restriction to $S$ is $k[\res^G _S Y]$-split. Since there is a $S$-map $f: X \to Y$, by \cite[Cor. 2.1.5]{Bouc-Tensor} the restricted $kS$-module sequence is also $\res^G _S X$-split. This gives that the sequence in (\ref{eqn:Y-sequence}) is also $X$-split

By definition $kY$ is $Y$-projective, so $k[\res^G _S Y]$ is $\res^G _S Y$-projective as a $kS$-module. Since there is an $S$-map $f': Y \to X$, by \cite[Cor. 2.2.4]{Bouc-Tensor} we can conclude that $k[\res^G _S Y]$ is also $\res^G _S X$-projective.  This implies that $kY$ is also $X$-projective, as one can easily see using part (iii) of Definition \ref{def:X_proj}.  

We have shown that the sequence in $(\ref{eqn:Y-sequence})$ is $X$-split with $X$-projective middle term.  Note that 
the sequence $$0 \to \Delta (X) \to kX \to k \to 0$$ is also $X$-split and the middle term $kX$ is $X$-projective.
Applying the Relative Schanuel's Lemma to these two sequences, we obtain  $$ \Delta (X) \oplus kY \cong \Delta (Y) \oplus kX.$$
Since $X^S=Y^S=\emptyset$, both $\Delta(X)$ and $\Delta(Y)$ are capped Dade $kG$-modules. From the isomorphism above and by the uniqueness of caps proved in  Proposition \ref{pro:UniqueCaps}, we obtain $$\Omega _X=[\Delta (X)]=[Cap (\Delta (X) )]=[Cap (\Delta (Y) )]=[\Delta(X)]=\Omega _Y,$$ and the result is proved.  
\end{proof}

The following result is an immediate consequence of Lemma \ref{lem:FixedPoints}.

\begin{corollary}\label{cor:Reduction} Let $X$ be a $G$-set with isotropy in $\cF_G$. If there is a decomposition $X=X_0\amalg X_1$ such that there is an $S$-map $X_0 \to X_1$, then $\Omega_X=\Omega_{X_1}$. 
\end{corollary}
 
This  shows that we may delete any $G$-orbit that maps into a distinct orbit without affecting the resulting element in the group of relative syzygies.
Another consequence of Lemma \ref{lem:FixedPoints} is the following.

\begin{corollary}\label{cor:Replace} Let $X$ be a finite $G$-set such that $X^S\neq\emptyset$, and let $Y$ be the $G$-set obtained by replacing each orbit $[G/H]$ of $X$ by $[G/T]$ for some $T\in\Syl_p(H)$.  Then $\Omega_X=\Omega_Y$ in $D(G)$.
In particular, if $H\leq G$ is a subgroup with $T\in \Syl_p(H)$ such that $T\notin\Syl_p(G)$, then $\Omega _{G/H} =\Omega_{G/T}$ in $D(G)$.
\end{corollary}

Corollary \ref{cor:Replace} allows us to restrict our attention to $p$-subgroup isotropy $G$-sets when considering the generators $\{\Omega_X\}$.
These two corollaries are used in the proof of  Proposition \ref{prop:Generates} to give a generating set for $D^{\Omega} (G)$.

The following is the analogue of \cite[Lem. 3.2.8]{Bouc-Tensor}. A similar result for strongly capped endo-$p$-permutation modules was also proven by Lassueur (see \cite[Lemma 5.7]{Lassueur-Dade}).

\begin{lemma}\label{lem:ShortExact}
Let $X$ be a $G$-set such that $X^S=\emptyset$, and let $0\to L\to k X\to N\to 0$ be an $X$-split short exact sequence of $kG$-modules.  Then
$L$ is a capped Dade module if and only if $N$ is a capped Dade module, in which case $[L]=\Omega_X+[N]$ in $D(G)$.
\end{lemma}

\begin{proof}
Tensoring the short exact sequence with $L^\ast$ yields 
\[
0\to L\otimes L^\ast\to k X\otimes L^\ast\to N\otimes L^\ast\to 0, 
\]
while first dualizing and then tensoring with $N$ yields
\[
0\to N^\ast\otimes N\to k X\otimes N\to L^\ast\otimes N\to 0.
\]
Both of these sequences are $X$-split and have $X$-projective middle term, so we may apply the Relative Schanuel's Lemma to obtain
\[
\End(L)\oplus(k X\otimes N)\cong\End(N)\oplus(k X\otimes L^\ast).
\]

The dual of the original sequence  $0\to N^\ast\to k X\to L^\ast\to 0$ is $X$-split, which gives  $\End(k X)\cong k X\otimes k X\cong(k X\otimes N^\ast)\oplus(k X\otimes L^\ast)$. Thus adding $k X\otimes N^\ast$ to both sides of the above isomorphism yields
\[
\End(L)\oplus(k X\otimes N)\oplus(k X\otimes N^\ast)\cong\End(N)\oplus\End(k X).
\]

If $N$ is a Dade module, then $\End(N)$ is $p$-permutation module, so our observation that $\End(L) \mid (\End(N) \oplus \End (kX) )$
shows that $\End(L)$ is also a $p$-permutation module.  Moreover, as $X^S=\emptyset$, every component of $\End(k X)\cong (kX)^* \otimes kX \cong k[X\times X]$ has a non-Sylow vertex, so the only Sylow-vertex components of the right hand side are those of $\End(N)$, which must be trivial as $N$ is a Dade module.  It follows that the only Sylow-vertex components of $\End(L)$ must be trivial as well, so $L$ is a Dade module. 
 
 If $N$ is a capped Dade module, then $k$ is a component of $\End(N)$. Note that $k$ cannot be a component of either $k X\otimes N^\ast$ or $k X\otimes N$ since these modules have components with non-Sylow vertices. This implies that $k$ is a summand of $\End(L)$, and hence $L$ must also be capped. The other direction of the statement can be proved in a similar way.
 
Assume now that both $L$ and $N$ are capped Dade $kG$-modules.  The exact sequence
\[
0\to\Delta(X)\to k X\to k\to 0
\]
is $X$-split, and it remains $X$-split after tensoring with $N$, so we may apply the Relative Schanuel's Lemma to the sequences
\[
%\xymatrix@R=0.6em@C=1.2em{
\xymatrix@R=1.6em@C=1.2em{
0\ar[r]&L\ar[r]&kX\ar[r]&N\ar[r]\ar@{=}[d]&0\\
0\ar[r]&\Delta(X)\otimes N\ar[r]&kX\otimes N\ar[r]&N\ar[r]&0}
\]
to obtain
\[
L\oplus(k X\otimes N)\cong k X\oplus(\Delta(X)\otimes N).
\]
Again, all components of $k X$, and hence $k X\otimes N$, have non-Sylow vertices, so $Cap(L)=Cap(\Delta(X)\otimes N)$, proving the last statement.
\end{proof}

Continuing to follow Bouc (cf. \cite[Lem. 5.2.1]{Bouc-Tensor}), we have the following:

\begin{lemma}\label{lem:Product}
If $X$ and $Y$ are $G$-sets such that $X^S=Y^S=\emptyset$, then
\[
\Omega_{X\amalg Y}+\Omega_{X\times Y}=\Omega_X+\Omega_Y
\]
in $D(G)$.
\end{lemma}

\begin{proof}
In light of Lemma \ref{lem:ShortExact}, it suffices to find an exact sequence
\[
0\to\Delta(X)\otimes\Delta(Y)\to k [X\times Y]\to\Delta(X\amalg Y)\to 0
\]
that is $k (X\times Y)$-split.  If we take the tensor product of the short exact sequences 
\begin{equation}\label{Eqn:X-sequence}
\xymatrix@R=.5em@C=1.2em{
0\ar[r]&\Delta(X)\ar[r]&kX\ar[r]&k\ar[r]&0}
\end{equation}
and 
\begin{equation}\label{Eqn:Y-sequence}
\xymatrix@R=.5em@C=1.2em{
0\ar[r]&\Delta(Y)\ar[r]&kY\ar[r]&k\ar[r]&0,
}
\end{equation}
then we obtain the long exact sequence
\begin{equation}\label{Eqn:LongSequence}
0 \to \Delta (X) \otimes \Delta (Y) \to kX\otimes kY \to kX\oplus kY \to k \to 0.
\end{equation}
We  claim that this long exact sequence is  $X \times Y$-split,  meaning that  it becomes a split sequence (as a long exact sequence) after it is tensored with $k[X\times Y]$.
To see this, note that  the short exact sequences (\ref{Eqn:X-sequence}) and (\ref{Eqn:Y-sequence}) are $X$-split and  $Y$-split, respectively, so their tensor product becomes a split exact sequence after we tensor with $kX \otimes kY$. This gives that the long exact sequence (\ref{Eqn:LongSequence}) is $X\times Y$-split since $k[X\times Y] \cong  kX \otimes kY$.

Note that $kX\oplus kY \cong k[X \amalg Y]$, hence the kernel of the map $kX\oplus kY \to k$ is isomorphic to $\Delta (X\amalg Y )$. Using the exactness of the sequence (\ref{Eqn:LongSequence}), we conclude that
\[
0\to\Delta(X)\otimes\Delta(Y)\to kX \otimes kY\to\Delta(X\amalg Y)\to 0
\]
is exact. This gives the desired short exact sequence.
\end{proof}

We are now ready to give a finite set of generators for $D^{\Omega} (G)$. A similar theorem was proved by Lassueur \cite[Lem. 12.1]{Lassueur-Dade}. Since we define $D^{\Omega} (G)$ using $G$-sets, the statements are slightly different, although the idea of the proof is the same.

\begin{proposition}\label{prop:Generates} 
Let $\cF_G$ denote the family of all non-Sylow $p$-subgroups of $G$ and let $\cF_G/G$ denote the set of $G$-conjugacy classes in $\cF_G$.
Then the set $\{\Omega_{G/ P}\ |\ [P] \in \cF_G/G \}$ generates $D^\Omega(G)$.
\end{proposition}

\begin{proof}
By Corollary \ref{cor:Replace}, we may assume that each point-stabilizer of $X$ is a non-Sylow $p$-subgroup of $G$. If $|S|=p^n$, we filter $D^\Omega(G)$ by
\[
0:=D^\Omega(G)_{-1}\leq D^\Omega(G)_0\leq D^\Omega(G)_1\leq\ldots\leq D^\Omega(G)_{n-1}=D^\Omega(G),
\]
where for $0\leq i\leq n-1$ we define $D^\Omega(G)_i$ to be the subgroup generated by those relative syzygies $\Omega_X$ for $X$ a $G$-set whose every point-stabilizer is a $p$-subgroup of order at most $p^i$.  We will prove inductively that each $D^\Omega(G)_i$ is generated by elements of the form $\Omega_{G/P}$ with $P\in \cF_G$.

To begin, the generating set for $D^\Omega(G)_0$ is those relative syzygies $\Omega_X$ where $X\cong m\cdot[G/1]$ is a free $G$-set.  Corollary \ref{cor:Reduction} then gives $X=\Omega_{G/1}$. Thus we have established the base case. 

Suppose that we have established the result for $k-1$.  Let 
\[
X=\coprod \limits_{i=1}^m G/P_i
\]
be such that for each $i$, the $p$-subgroup of $P_i$ is of order at most $p^k$.  By Corollary \ref{cor:Reduction}, we may assume that there are no $S$-maps from $G/P_i$ to $G/P_j$ for $i\neq j$, as deleting the orbit $G/P_i$ does not affect the resulting relative syzygy in $D^\Omega(G)$. We now induct on $m$, the base case $m=1$ being the desired conclusion.

Write $X=(G/P_1) \amalg X'$ where $X'=\coprod\limits_{i=2}^m G/P_i$.  Lemma \ref{lem:Product} gives us
\[
\Omega_X=\Omega_{[(G/P_1) \amalg X']}=\Omega_{G/P_1}+\Omega_{X'}-\Omega_{[(G/P_1) \times X']}.
\]
By the second induction (on $m$), we have that $\Omega_{X'}$ is a sum of $\{ \Omega _{G/Q_j} \}$ for some $p$-subgroups $Q_j$. It therefore suffices to show that  $\Omega_{[(G/P_1) \times X']} $ can also be expressed as a sum of $\Omega _{G/Q_j}$ for some $p$-subgroups $Q_j $.

Consider the point-stabilizer of an element in $(G/P_1)\times X'$, which is the intersection of the point-stabilizers of the two component elements.  In other words, the point-stabilizer of $(gP_1,x')$ is $^gP_1\cap G_{x'}$.  From our assumption that there are no $S$-maps from $G/P_i$ to $G/P_j$ when $i\neq j$, we must have that $P_1$ is not $G$-subconjugate to any of the point-stabilizers of $X'$. Therefore the point-stabilizers of $(G/P_1) \times X'$ are all properly contained in conjugates of $P_1$, and thus have order at most $p^{k-1}$. This gives that $\Omega_{[ (G/P_1)\times X']}\in D^\Omega(G)_{k-1}$, and our inductive assumption on $k$ completes the proof.
\end{proof}

%-------------------------------

\section{Superclass functions and the Bouc homomorphism}\label{sect:Superclass}
 
Let $G$ be a finite group. A set of subgroups of $G$ is called a \emph{family}  if it is closed under $G$-conjugation and taking subgroups. We denote an arbitrary family of subgroups by $\cH$. As before, $\cF_G$ denotes the family  of all non-Sylow $p$-subgroups of $G$, and $\cF_p$ denotes the family of  all $p$-subgroups of $G$.  

\begin{definition}
A \emph{superclass} function of $G$ defined on a family $\cH$ is a function $f: \cH \to\ZZ$ that is constant on each $G$-conjugacy class.
We write $C(G, \cH)$ for the set of superclass functions defined on $\cH$. Note that $C(G, \cH)$ is an abelian group under pointwise addition. 
\end{definition}

\begin{remark} For every family $\cH$ of subgroups of $G$, the relative Burnside ring $B(G, \cH)$ is defined as the subring of the Burnside ring generated by those $G$-sets whose point-stabilizers lie in $\cH$.  The group of superclass functions $C(G, \cH)$  can be identified with the $\ZZ$-dual of the relative Burnside ring $B^{\vee} (G, \cH ):=\Hom (B(G, \cH ), \ZZ)$ by the assignment
\[
B^\vee(G,\cH)\to C(G,\cH):f\mapsto(H\mapsto f(G/H))
\]
for all $H\in\cH$. We sometimes use this identification without referring to it explicitly.
\end{remark}

Let $\cH/G$ denote the set of $G$-conjugacy classes of subgroups in $\cH$. There is a canonical $\ZZ$-basis $\{\delta_{[H]}\ |\ [H] \in \cH /G \}$ for $C(G, \cH)$, whose elements are defined by 
%\[
%\delta_{[H]} (K) =\left\{\begin{array}{ll}
%1&\text{if }\  [K]=[H]\\
%0& \text{if }\ [K]\neq  [H]
%\end{array}\right. 
%\]
$$
\delta_{[H]} (K) =\begin{cases}
1\  \text{ if }\  [K]=[H] \\
0\  \text{ if }\ [K]\neq  [H]
\end{cases}
$$
%$$\omega _X (P) =\begin{cases} 1 \ \text{  if   }  \ X^P \neq \emptyset, \\ 0 \  \text{ otherwise.} \end{cases} $$
for all $K\in\cH$. 

For every $G$-set $X$, let $\omega_X\in C(G, \cH)$ denote the superclass function defined by 
%\[
%\omega_X(K)=\left\{\begin{array}{ll}
%1& \text{if  } \ X^K\neq\emptyset\\
%0& \text{if } \ X^K=\emptyset
%\end{array}
%\right.
%\]
$$
\omega_X(K)=\begin{cases}
1 \  \text{ if } \ X^K\neq\emptyset \\
0 \ \text{ if } \ X^K=\emptyset
\end{cases}
$$
for all $K\in \cH$. We have the following observation due to Bouc  \cite[Lemma 2.2]{Bouc-Remark}.
 
\begin{lemma}\label{lem:Basis} 
Let $\cH$ be any family of subgroups in $G$.
Then the  set $\{\omega_{G/H}\ |\ [H] \in \cH /G\}$ is a $\ZZ$-basis for $C(G, \cH)$.
\begin{proof} For each $H \in \cH$, we can write 
\[
\omega_{G/H}=\sum \limits _{[K]\leq [H]}  \delta _{[K]},
\]
where the partial order is induced by $G$-subconjugacy. %It follows that $\{\omega_{[G/H]}\}$ is a $\QQ$-basis for $\QQ\otimes C(G,\cH)$. 
By totally ordering $\cH/G$ appropriately, it is easy to see that the matrix for the linear transformation that takes  $\{ \delta _{[K]} \}$ to $\{\omega _{G/H} \}$  can be made an upper triangular matrix with integral entries and all diagonal terms equal to 1. The inverse matrix must therefore be an integer matrix of the same form, and the result follows.
\end{proof}
\end{lemma}

 It is easy to see that the superclass functions $\{ \omega_X\}$, as $X$ ranges over all $G$-sets, satisfy the following relations:

\begin{lemma}\label{lem:FixedRelations}
Let $X$ and $Y$ be $G$-sets. If for every subgroup $H \in \cH$, the fixed point set $X^H \neq \emptyset$ precisely when $Y^H \neq \emptyset$, then   $\omega_X=\omega_Y$ in $C(G, \cH)$.
\end{lemma}

\begin{proof} Immediate from the definitions of $\omega_X$ and $\omega_Y$.
\end{proof}

We also have a relation coming from a product of two $G$-sets.

\begin{lemma}\label{lem:ProductRelations} If $X$ and $Y$ are two $G$-sets, then
\[
\omega_{X\amalg Y}+\omega_{X\times Y}=\omega_X+\omega_Y
\]
in $C(G, \cH)$.
\end{lemma}
 
\begin{proof}
Fix some subgroup $H \in \cH$.
There are three cases to consider: both $X$ and $Y$ have an $H$-fixed-point, exactly one has a $H$-fixed-point, and neither has a $H$-fixed-point.  Since $(X\times Y)^H=X^H \times Y^H$ and $(X\amalg Y)^H=X^H \amalg Y^H$, these cases correspond respectively to:  both $X\times Y$ and $X\amalg Y$ have an $H$-fixed-point, $X\times Y$ has no $H$-fixed-point but $X\amalg Y$ does, and neither $X\times  Y$ nor $X\amalg Y$ has an $H$-fixed-point.  The result is now immediate.
\end{proof}

Note that the conclusions of Lemmas \ref{lem:FixedRelations} and \ref{lem:ProductRelations}  are the direct analogues of Lemmas \ref{lem:FixedPoints} and \ref{lem:Product},  with $\omega_X$ in place of $\Omega_X$. 
As before (cf. Corollaries \ref{cor:Reduction} and \ref{cor:Replace}), Lemma \ref{lem:FixedRelations} has the following consequences:

\begin{corollary}\label{cor:Relations2}
Let $\cH$ be a family of $p$-subgroups of $G$ and $X$ a finite $G$-set. 
\begin{enumerate}
\item\label{cor:Relations2_i} If there is a decomposition $X \cong X_0\amalg X_1$ such that 
there is an $S$-map $X_0 \to X_1$, then $\omega_X=\omega_{X_1}$ in $C(G, \cH)$. 
\item\label{cor:Relations2_ii}  If $Y$ is formed from $X$ by replacing each orbit $[G/H]$ by $[G/T]$ for $T\in\Syl_p(H)$, then $\omega_X=\omega_Y$ in $C(G,\cH)$. In particular, for any $H\in\cH$ and $T\in \Syl_p(H)$, we have $\omega _{G/H} =\omega_{G/T}$ in $C(G, \cH)$.  
\end{enumerate}
\end{corollary}

The assignment $\omega_{G/P}\mapsto\Omega_{G/P}$ induces the group homomorphism
\[
\overline\Psi_G:C(G,\cF_G)\to D^\Omega(G). 
\]
%:\sum_{[P]\in\cF_p/G}a_{[P]}\cdot\omega_{[G/P]}\mapsto\sum_{[P]\in\cF_p/G}a_{[P]}\cdot\Omega_{[G/P]}.
%that sends $\omega_{[G/P]}$ to $\Omega_{[G/P]}$ for all $P \in \cF_G$.
As  $\{\omega_{G/P}\ |\ [P]\in\cF_p/G\}$ forms a free $\ZZ$-basis for $C(G,\cF_p)$, this rule gives a well-defined map, and since 
$\{\Omega_{G/P}\ |\ [P]\in\cF_G/G\}$ forms a (nonfree) generating set of $D^\Omega(G)$ by Proposition \ref{prop:Generates}, $\overline\Psi_G$ is a surjection. To extend this homomorphism to $C(G, p)$ we declare $\Omega _{G/S}=0$ and define $$\Psi _G  : C(G, p) \to D^{\Omega } (G)$$ 
to be the homomorphism that sends  $\omega _{G/P}$ to $\Omega _{G/P}$ for every $[P] \in \cF_p/G$. This homomorphism is still well-defined and surjective.
Since $\Psi _G$ is a generalization of the homomorphism defined for $p$-groups by Bouc in \cite[Theorem 1.7]{Bouc-Remark}, we refer to it as the \emph{Bouc homomorphism}.

The declaration of $\Omega _{G/S}=0$ suggests that we can extend the definition of $\Omega_X$ to arbitrary $G$-sets $X$ as follows:

\begin{definition} For any $G$-set $X$, set $$\Omega _X=\begin{cases} [\Delta(X)]  &\text{  if   } X^S = \emptyset, \\ 0 & \text{ if } X^S\neq \emptyset.  \end{cases}$$ 
\end{definition}

Note that this definition is consistent with the convention that is used when $G=S$ is a $p$-group.  In this earlier situation, when $|X^S|>1$ we have $\Delta(X)$ is a capped endo-permutation module with cap $k$, so we have $\Omega_X=[\Delta(X)]=0$ in $D(X)$ by definition. However if $|X^S|=1$ then $\Delta(X)$ is not a capped endo-permutation module.  In this case $\Omega_X$ was taken to be trivial by convention.

In the general finite group situation, if $X^S\neq\emptyset$ then $\Delta(X)$ may contain non-isomorphic Sylow-vertex summands, and hence it may not be a Dade module.  Our choice of setting $\Omega_X=0$ in $D^\Omega(G)$ when $X^S\neq\emptyset$ is therefore not only consistent with the $p$-group situation, it is the only reasonable generic option available.   

\begin{remark} In the non-$p$-group case one should be careful when using this convention, as it is sometimes possible that $\Delta(X)$ may be a Dade module with $[\Delta (X)]\neq 0$ even though $X^S\neq\emptyset$. In this case we still have $\Omega_X=0$. For example, when $G=C_2$ and $p\neq 2$ we have $\Delta (G/1) \cong k(-1) $ and $[k(-1)]\neq 0$ in $D(G)$, however we still take $\Omega _{G/1}=0$. In fact, for this group we have $D_k (C_2) \cong \ZZ/2$ but $D^{\Omega } _k (C_2)=0$ when $k$ is a field of odd characteristic.
\end{remark}

With this convention we have the following theorem.

\begin{proposition}\label{pro:BoucHom} The Bouc homomorphism $$\Psi _G: C(G, p) \to D ^{\Omega} (G)$$  sends $\omega_X$ to $\Omega_X$ for every $G$-set $X$.
\end{proposition} 

\begin{proof} If $X$ is a $G$-set such that $X^S \neq \emptyset$, then $\omega _X : \cF_p \to \ZZ$ is the constant superclass function with value $1$. 
In particular, we have $\omega _X=\omega _{G/S}$. Then $\Psi _G (\omega _X)=\Psi _G (\omega _{G/S} ) =0 $. Since in this case  by definition $\Omega _X=0$, we have obtained the desired conclusion $\Psi _G (\omega_X)=\Omega _X$. 

Now assume that $X^S =\emptyset$. By Corollaries \ref{cor:Replace} and \ref{cor:Relations2}\ref{cor:Relations2_ii}, we may take the point-stabilizers of $X$ to be $p$-subgroups.
As in the proof of Proposition \ref{prop:Generates}, we induct on the maximum order of the point-stabilizers of $X$. 
In the case where this order is $1$, we must have that $X$ is a free $G$-set, so $X=m\cdot[G/1]$. We then have $\Omega_X=\Omega_{G/1}$ and $\omega_X=\omega_{G/1}$ by Corollaries \ref{cor:Reduction}  and  \ref{cor:Relations2}\ref{cor:Relations2_i}, respectively, so it follows that $\Psi_G(\omega_X)=\Omega_X$ by the definition of $\Psi_G$. We have established our base case.

Suppose now that the result holds for all $G$-sets with point-stabilizers $p$-groups of order less than $p^k$.  Among all $G$-sets having a point-stabilizer  of order $p^k$, we induct on the number of orbits $m$, the case $m=1$ being a transitive $G$-set and therefore covered by the definition of $\Psi_G$. 

Now let $X$ be such a $G$-set with $m$ orbits and assume the result has been proved for all values less than $m$.  If there are $G$-orbits $[G/P]$ and $[G/Q]$ in $X$ such that $P$ is $G$-subconjugate to $Q$, and if $X'$ is $X$ with one copy of $[G/P]$ deleted, Corollaries \ref{cor:Reduction} and \ref{cor:Relations2}\ref{cor:Relations2_i} give $\Omega_X=\Omega_{X'}$ and $\omega_X=\omega_{X'}$, respectively, in which case the result is proved by the inductive hypothesis on $m$.  We may therefore assume that no two orbits of $X$ have point-stabilizers that are comparable by the $G$-subconjugacy relation. 

Finally, we may write $X=[G/P]\amalg X'$ where $|P|=p^k$ and $X'$ has $m-1$ orbits, each of which has point-stabilizers of order at most $p^k$ and is not $G$-subconjugate to $P$.  Writing
\[
\Omega_X=\Omega_{G/P}+\Omega_{X'}-\Omega_{[(G/P)\times X']}\qquad\textrm{and}\qquad
\omega_X=\omega_{G/P}+\omega_{X'}-\omega_{[(G/P)\times X']}
\]
by Lemmas \ref{lem:Product} and \ref{lem:ProductRelations}, we note that we already have $\Psi_G(\omega_{G/P})=\Omega_{G/P}$ by definition and $\Psi_G(\omega_{X'})=\Omega_X'$ by our induction on $m$.  On the other hand, the assumption that no point-stabilizer of $X'$ is $G$-subconjugate to $P$ implies that all point-stabilizers of $(G/P) \times X'$ are of order at most $p^{k-1}$, so we also have $\Psi_G(\omega_{[(G/P)\times X']})=\Omega_{[(G/P)\times X']}$ by the induction assumption.  As $\Psi_G$ is linear, we have proved $\Psi_G(\omega_X)=\Omega_X$, thereby completing the induction and the proof.
\end{proof}
 
Note that  Proposition \ref{pro:BoucHom} completes the proof of Theorem \ref{thm:IntroBoucHom}.

One of the consequences of the existence of the Bouc homomorphism is that any equation that hold for $\omega_X$ also holds 
for $\Omega_X$.  Using this, we can conclude the following:

\begin{corollary} Let $G$ be a finite group and $X$ be a $G$-set. Then 
$$\Omega _X =\sum _{  \substack{[Q], [P] \in \cF_p /G,\\ Q\leq _G P,\ X^P\neq \emptyset} } \mu_G (Q, P ) \Omega _{G/Q} $$
in $D^{\Omega} (G)$, where $\mu_G$  denotes the M\" obius function of the poset of $G$-conjugacy classes of $p$-subgroups in $G$.
\end{corollary}

\begin{proof} Repeating the arguments given in \cite[2.3, 2.6]{Bouc-Remark}, one can  easily see that the above  formula holds for $\omega_X$, and hence it also holds for 
$\Omega_{X}$.
\end{proof}

%--------------------------------------

\section{Moore $G$-spaces with non-Sylow $p$-subgroup isotropy}\label{sect:GMoore}
 
 In this section we consider Moore $G$-spaces. We show that if all the point-stabilizers of a Moore $G$-space $X$ are non-Sylow $p$-subgroups of $G$, then the top reduced homology group of $X$ with coefficients in $k$ is a Dade $kG$-module (Theorem \ref{thm:Moore}). This result will be extended in the next section to all Moore $G$-spaces, with minor modifications. All this is ultimately aimed at showing that the dimension function of a $k$-oriented real representation sphere is in the kernel of the Bouc homomorphism.

Let $G$ be a finite group and $\cH$ a family of subgroups of $G$ (closed under conjugation and taking subgroups). The \emph{orbit category of $G$ with respect to the family $\cH$} is the category $\cO_{\cH} (G)$ whose objects are the transitive $G$-sets $\{G/H\ |\ H\in\cH\}$ and whose morphisms from $G/H$ to $G/K$ are given by the $G$-maps $G/H \to G/K$. When the group $G$ is clear from the context, we will write $\OH$ for $\cO _{\cH} (G)$. 
 
Let $R$ be a commutative ring with unity. An \emph{$R\OH$-module} $M$ is a contravariant functor from the category $\OH$ to the category of $R$-modules. The value of an $R\OH$-module $M$ at $G/H$ is written $M(H)$. By identifying $\Aut _{\OH } (G/H)$ with $\overline N_G (H) :=N_G(H)/H$, we view the $R$-module $M(H)$ as an $R \overline N_G(H)$-module. In particular, $M(1)$ is an $RG$-module. Note that for every $H\in \cH$, the subgroup $H$ acts trivially on $M(H)$. 
%The category of $R\OH$-modules is an abelian category, so the usual concepts of projective module, exact sequence, and chain complexes are available. 
For more information on modules over orbit categories,  
we refer the reader to \cite[\S 9, \S 17]{Lueck-Book} and \cite{HPY}.

\begin{definition}\label{def:Free} For a $G$-set $X$, the $R\OH$-module $R[X^?]$ is defined by the rule
\[
G/H\mapsto R[X^H] \cong R\Hom_G(G/H,X),
\]
where $X^H$ denotes the $\overline{N}_G(H)$-set of $H$-fixed points. 
\end{definition}

An $R\OH$-module is called \emph{free} if it is isomorphic to a direct sum of modules of the form $R[(G/K)^?]$ with $K \in \cH$. By the Yoneda Lemma, every free $R\OH$-module is projective (see \cite[Section 2A]{HPY}). 

%Note an $R\OH$ module of the form $R[X^?]$ is free if and only if all point stabilizers of $X$ lie in the family $\cH$.
%By the Yoneda lemma, every free $R\OH$-module is projective (see \cite[Section 2A]{HPY}). 

More generally, given a $G$-CW-complex $X$, we can associate to it a chain complex of $R\OH$-modules.

\begin{definition}
Let $X$ be a finite $G$-CW-complex.  For each $i\geq 0$, consider the $R\OH$-module
\[
C_i(X^?;R): G/H\mapsto C_i(X^H;R),
\]
where $C_i( X^H; R)$ denotes the free $R$-module whose basis is the set of $i$-dimensional cells in the $H$-fixed subspace $X^H$.  These $R\OH$-modules are connected by the CW-boundary maps to yield a chain complex of $R\OH$-modules
$$C_* (X^?, R) : \cdots \to C_i (X^? ; R) \maprt{\partial _i} C_{i-1} (X^? ; R) \to \dots \to C_0 (X^? ; R) \to 0.$$
\end{definition}

Note that for each $H \in \cH$, the complex $C_* (X^H; R)$ is the chain complex of the subspace $X^H$. This gives an interpretation of $C_* (X^?; R)$ as a functor from the category $\OH$ to the category of chain complexes over $R$. 

For each $i\geq 0$, $C_i (X; R) \cong R[X_i ^?]$ where $X_i$ denotes the $G$-set of all $i$-dimensional cells in $X$. From this it is easy to see that the chain complex of a $G$-CW-complex $X$ is a chain complex of free $R\OH$-modules if all point-stabilizers of $X$ lie in $\cH$.

Let $\underline{R}$ denote the constant functor with values $\underline{R} (H)=R$ for every $H \in \cH$.
For each $G$-map $f: G/K \to G/H$  the induced map $f^* \colon \underline{R}(H) \to \underline{R} (K)$ is taken to be the  identity map $\id : R \to R$. For every free $R\OH$-module $R[X^?]$ there is an augmentation map $\varepsilon : R[X^?] \to \underline R$, defined object-wise by the augmentation map $R[X^H ] \to R$.

If $C_* (X^?; R)$ is a chain complex of free $R\OH$-modules, then there is an augmentation map $C_0 (X^?; R) \to \underline R$. Appending this map to the chain complex $C_* (X ^?; R)$ gives the augmented complex $\widetilde C_* (X^?; R)$ whose $i$-th homology is an $R\OH$-module. This $R\OH$-module will be written $\widetilde H_i (X^? ; R)$, and is called the \emph{$i$-th reduced homology module of $X$}.  
 
Recall that a $G$-CW-complex is called \emph{finite} if it has finitely many cells.
 
\begin{definition}\label{def:MooreComplex} 
A \emph{Moore $G$-space over $R$, relative to the family $\cH$,} is a finite $G$-CW-complex $X$ such that for every $H\in\cH$ there is exactly one $i\geq 0$ such that $\widetilde H_i(X^H;R)\neq 0$. If $X$ is a Moore $G$-space then there is a superclass function $\underline n:\cH\to\ZZ$ such that $\widetilde H_i(X^H;R)$ vanishes for all $i\neq \underline{n}(H)$. We sometimes record these data by referring to $X$ as an \emph{$\underline n$-Moore $G$-space}.
\end{definition}

In our applications, we often need the homological dimension of the fixed point subspace $X^H$ to coincide with its geometric dimension as a $CW$-complex, denoted by $\dim (X^H)$. We define the following condition for Moore $G$-spaces to guarantee this.

\begin{definition}
The $\underline n$-Moore $G$-space $X$ is \emph{tight} if $\underline n(H)=\dim (X^H)$ for all $H\in\cH$.
\end{definition}

We have many examples of tight Moore $G$-spaces. 

\begin{example} Let $V$ be a real $G$-representation and $X:=S(V)$ the unit sphere in $V$ with respect to some $G$-equivariant norm. The $G$-space $X$ is a smooth $G$-manifold, which can be triangulated to obtain a $G$-CW-complex. In this case $X^H=S(V^H)$ is a sphere for every $H\leq G$, hence $\widetilde H_i (X^H; R) =0$ when $i \neq \dim X^H$. Thus $X$ is a tight $\underline{n}$-Moore $G$-space over $R$ where $\underline{n} (H)=\dim_{\RR} (V^H )-1$.    
\end{example}

%Another source of examples of Moore $G$-spaces is discrete $G$-spaces.

\begin{example} Let $X$ be a finite $G$-set, considered as a zero dimensional $G$-CW-complex. The augmented complex $\widetilde C_* (X^?; R)$ is  of the form
$$0 \to  R[X^?] \maprt{\varepsilon}\underline R \to 0$$ where $\varepsilon$ is the augmentation map. If $X^H=\emptyset $, then $\widetilde H_i ( X^H; R) \cong R$ for $i=-1$ and vanishes in all other dimensions. If $X^H\neq \emptyset$, then we have a short exact sequence $$0 \to \Delta (X^H) \to k[X^H ] \to R \to 0,$$
so the reduced homology of $X^H$ is concentrated in dimension 0 and $\widetilde H_0 (X^H; R) \cong \Delta (X^H)$ as an $R \overline N_G(H)$-module. We conclude that $X$ is a tight $\underline{n}$-Moore $G$-space over $R$, where $\underline{n}=\omega_X-1$.  
\end{example}

\begin{definition}\label{def:DimFunctionMoore} The \emph{dimension function} of an $\underline{n}$-Moore space $X$ is the superclass function $\Dim (X) : \cH \to \ZZ$ defined by 
$\Dim (X) (H)=\underline{n} (H) +1$. 
\end{definition}

For every finite $G$-set $X$, the superclass function $\omega_X$ is realized as the dimension function of the discrete $G$-space $X$.
To realize sums of multiple such $\omega_X$, one can use the join construction for $G$-Moore spaces.

Given two $G$-CW-complexes $X$ and $Y$, the join $X\ast Y$ is defined as the quotient space $X\times Y \times [0,1] /\sim $ with the identifications $(x,y,1)\sim (x',y,1)$ and $(x,y,0)\sim (x, y', 0)$ for all $x,x'\in X$ and $y,y'\in Y$. The $G$-action on $X\ast Y$ is given by $g(x,y,t)=(gx, gy, t)$ for all $x\in X$, $y\in Y$, and $t\in [0,1]$.  The join $X*Y$ has a natural $G$-CW-complex structure, once we assume that the topology on the product  is the compactly generated topology %to avoid the usual problems with products of spaces 
(see \cite[Section 4]{Yalcin-Moore} for more details).

\begin{lemma}\label{lem:DimOfJoin} If $X$ and $Y$ are Moore $G$-spaces, then $X\ast Y$ is a Moore $G$-space with dimension function $\Dim (X*Y)=\Dim(X)+\Dim(Y)$. Moreover if $X$ and $Y$ are tight Moore $G$-spaces, then $X\ast Y$ is also tight. 
\end{lemma}

\begin{proof} We have  $(X\ast Y)^H = X^H \ast Y^H$ for every $H \leq G$. Since the join of two Moore spaces is a Moore space the result follows.
\end{proof}

As a consequence we obtain:
 
\begin{proposition}\label{pro:JoinGSet}
Let $X$ be the join of a finite set of finite $G$-sets $\{X_i\}$.  Then $X$ is a tight Moore $G$-space and  $\Dim (X)= \sum _i \omega _{X_i}$. 
In particular, every superclass function $f=\sum a_P \cdot \omega _{G/P}$ with $a_P\geq 0$ can be realized as the dimension function of a tight Moore $G$-space.
\end{proposition}

\begin{proof} 
We can define a $G$-simplicial complex structure on $X$ using the following:
Define a relation on $\amalg _i X_i$ by declaring $x\leq y$ if $x\in X_i$ and $y \in X_j$ for some $i<j$. The topological realization of this poset
is a $G$-simplicial complex which is homeomorphic to the join $\ast_i X_i$. It is clear that the dimension function of  $X$ is the function $\sum _i \omega _{X_i}$. 
\end{proof}

For a Moore $G$-space $X$ constructed as the join of $G$-sets $\{ X_i \}$, 
the augmented simplicial chain complex $\widetilde C_* (X ^?; R) $ is the tensor product of chain complexes $$0 \to R[X_i^?] \to \underline R \to 0.$$  
Hence the homology of  the complex $\widetilde C_* (X; R) $ is isomorphic to the tensor product  $\bigotimes _i \Delta (X_i)$ as an $RG$-module.
If all isotropy subgroups of the $G$-sets $X_i$ are non-Sylow $p$-subgroups and $R=k$ is a field of characteristic $p>0$, then by 
Proposition \ref{pro:RelativeDade}, each $\Delta (X_i)$ is a capped Dade module, hence in this case the reduced homology of $X$ is a Dade $kG$-module whose class in the Dade group is the sum $\sum _i \Omega_{X_i}$. 

This observation generalizes to all Moore $G$-spaces satisfying the same isotropy subgroup condition.

\begin{theorem}\label{thm:Moore} Let $G$ be a finite group and $k$ be a field of characteristic $p>0$. Suppose that $X$ is 
an $n$-dimensional tight Moore $G$-space over $k$ such that all the point-stabilizers of $X$ are non-Sylow $p$-subgroups. 
Let $X_i$ denote the $G$-set of $i$-dimensional cells of $X$. Then  $\widetilde H_{n} (X, k)$ is a capped Dade 
$kG$-module, and $$[\widetilde H_{n} (X; k)]=\sum _{i=1}^n \Omega _{X_i}$$ in $D ^{\Omega} (G)$.
\end{theorem}

\begin{proof} We follow the argument given in \cite{Yalcin-Moore} with some modifications. Let $\cH=\cF_G$ be the family of all non-Sylow $p$-subgroups in $G$. Let $\un{n}: \cH \to \ZZ $ be a superclass function and $X$ be an $\underline{n}$-Moore $G$-space with isotropy in $\cH$. The augmented chain complex $\widetilde C_* (X^?; k)$ is a chain complex of free $k\OH$-modules of the form
$$ 0 \to k[X_n^?]\maprt{\bd_n}  \cdots \maprt{}  k[X_i ^?] \maprt{\partial_i} k[X_{i-1} ^?] \maprt{} \cdots \maprt{} k[X_0 ^?] \maprt{\bd_0 } \underline k \to 0.$$
The evaluation of this sequence at the trivial subgroup $Q=1$ gives an exact sequence of $kG$-modules
$$ 0 \to \widetilde H_n (X; k) \to k[X_n]\maprt{\bd_n}  \cdots \maprt{}  k[X_i] \maprt{\partial_i} k[X_{i-1}] \maprt{} \cdots \maprt{} k[X_0] \maprt{\bd _0} k \to 0.$$

If we show that for each $i \geq 0$, the short exact sequence
\begin{equation}\label{eqn:X_iSequence}
0 \to \ker \partial _i \to k[X_i] \to \mathrm{im}\, \partial _i \to 0 
\end{equation}
is $X_i$-split then the conclusion will follow from repeated applications of Lemma \ref{lem:ShortExact}. To show that the sequence in (\ref{eqn:X_iSequence}) is $X_i$-split, we consider the chain complex of $R\OH$-modules  
\[
\bD:= 0 \to k[X_i ^?] \otimes k[X_i ^?] \xrightarrow{\partial _i \otimes id}  k[X_{i-1} ^? ]\otimes k[X_i^?] \to \cdots \to k[X_0 ^?] \otimes k[X_i ^? ] \to k[X_i ^? ]\to 0
\]
obtained by first tensoring the augmented complex $\widetilde C_* (X^? ; k)$ with the free module $k[X_i ^?]$ and then truncating at dimension $i$. Note that for each $j\geq 0$, the $k\OH$-module 
$$\bD_j = k[X_j ^?]\otimes k[X_i ^?] \cong k[(X_j \times X_i)^?]$$ 
is free, hence $\bD$ is a chain complex of projective $k\OH$-modules.  

We claim that $\bD$ has no homology in dimensions strictly less than $i$. For each $Q \in \cH$ such that $X_i ^Q=\emptyset$, the complex $\bD(Q)$ is identically zero. If $Q \in \cH$ is such that $X_i ^Q \neq \emptyset$, then $\underline{n} (Q)= \dim X^Q \geq i$ by the tightness of $X$. This implies that $\widetilde H_j (X^Q ; R)=0$ for all $j< i$, meaning that the chain complex 
$$ 0 \to k[X_i ^Q] \maprt{\partial _i}  k[X_{i-1} ^Q ] \to \cdots \to k[X_0 ^Q] \to k \to 0$$    
has no homology  in dimensions $j< i$. Since tensoring an exact chain complex with $k[X_i ^Q]$ over $k$ does not change the exactness of a sequence,  we conclude that $$H_j (\bD) (Q)=H_j (\bD(Q))=0$$ for all $j<i$ and $Q \in \cH$. Hence the $k \OH$-module $H_j (\bD)$ vanishes 
for all $j<i$.

To complete the proof, observe that the condition $H_j (\bD)=0$ for $j<i$ implies that $\mathrm{im} (\partial _i \otimes \id )$ is a projective $k\OH$-module.
Hence the short exact sequence
$$0 \to \ker (\partial _i \otimes \id) \to k[X_i ^?] \otimes k[X_i ^?] \to \mathrm{im} (\partial _i \otimes \id ) \to 0$$
splits as an exact sequence of $k\OH$-modules. This implies that the sequence  
$$0 \to \ker (\partial _i)  \otimes k[X_i]  \to k[X_i] \otimes k[X_i] \to \mathrm{im} (\partial _i) \otimes k[X_i] \to 0$$
splits as a sequence of $kG$-modules. Hence for each $i$, the sequence 
$$ 0 \to \ker \bd _i \to k[X_i] \to \mathrm{im} (\bd _i ) \to 0$$
is $X_i$-split.
% Evaluating this sequence at the subgroup $Q=1$ gives a split exact sequence of $kG$-modules. 
%Hence the sequence in (\ref{eqn:X_iSequence}) is $X_i$-split. 
This completes the proof.
\end{proof}

A $G$-CW-complex is called \emph{full} if for every $H \leq G$, we have $C_i (X^H)\neq 0$ whenever $i \leq \dim (X^H)$. This property always holds when $X$ is a $G$-simplicial complex. By the equivariant simplicial approximation theorem, every $G$-CW-complex $X$ is $G$-homotopy equivalent to a $G$-simplicial complex, and hence $G$-homotopy equivalent to a full $G$-CW-complex $Y$. Recall that two $G$-spaces $X, Y$ are \emph{$G$-homotopy equivalent} if there are $G$-maps $f: X\to Y$ and $f': Y \to X$ such that the compositions $f' \circ f$ and $f \circ f'$ are $G$-homotopic  to the identity maps. 

For a full complex we can prove the following.

\begin{lemma}\label{lem:DimSum} Let $X$ be a tight Moore $G$-space of dimension $n$, relative to a family $\cH$. Suppose that $X$ is a full complex, and  $X_i$ denotes the $G$-set of $i$-dimensional cells of $X$.  Then $$\Dim(X)=\sum _{i=0} ^n \omega _{X_i}$$ in $C(G, \cH)$.
\end{lemma}

\begin{proof} Fix $H \in \cH$. The sum $\sum _i \omega _{X_i} (H)$ is the number of indices $i$ such that $X_i ^H \neq \emptyset$. Since $X_i ^H \neq \emptyset$ if and only if $i$ satisfies $0 \leq i \leq \dim X^H$, we obtain that $$\sum _{i=0} ^n  \omega _{X_i} (H)=1+\dim (X^H)=\Dim (X) (H).$$
\end{proof}
 
 We can define a group of Moore $G$-spaces for a finite group $G$, as was done in \cite{Yalcin-Moore} for $p$-groups.  
  
 \begin{definition}\label{defn:equivalence} 
The Moore $G$-spaces $X$ and $Y$ are \emph{equivalent}, denoted $X \sim Y$, if  $X$ and $Y$ are $G$-homotopy equivalent.   The equivalence class of a Moore $G$-space $X$ will be written $[X]$.
\end{definition}

It is easy to show that if $X \sim X'$ and $Y\sim Y'$, then $X* Y \sim X' * Y'$. Hence the join operation defines an addition of the equivalence classes of Moore $G$-spaces given by $[X]+[Y]=[X*Y]$. Note that if the isotropy subgroups $X$ and $Y$ lie in $\cH$, then the isotropy subgroups of $X*Y$ also lie in $\cH$. The set of equivalence classes of Moore $G$-spaces with isotropy in $\cH$ with this addition operation is a commutative monoid, so we can apply the Grothendieck construction to define the group of Moore $G$-spaces. 

\begin{definition}\label{defn:GroupMoore}
Let $G$ be a finite group and $k$ be a field of characteristic $p>0$. The \emph{group of tight Moore $G$-spaces} $\cM _t (G, \cH)$ is  the Grothendieck group of $G$-homotopy classes of tight Moore $G$-spaces with isotropy in $\cH$ with addition defined by 
$[X]+[Y] :=[X \ast Y]$. 
\end{definition}

Note that since every $G$-CW-complex $X$ is $G$-homotopy equivalent to a full $G$-CW-complex we can define a homomorphism 
\[
\Dim : \cM _t (G, \cH) \to C(G, \cH): [X]-[Y] \mapsto\Dim (X) -\Dim (Y).
\]
We call this the \emph{dimension homomorphism}. 
Note that the dimension homomorphism is surjective since $C(G, \cH)$ is generated by the $\{ \omega_X \}$, and $\omega _X$ is the dimension function of a discrete $G$-space $X$, which is a tight Moore $G$-space.

As before let $\cF_G$ denote the family of all non-Sylow $p$-subgroups of $G$. By Theorem \ref{thm:Moore}, if $X$ is a tight Moore $G$-space relative to the family $\cF_G$, then its reduced homology is a Dade module. Using this, we can define an homomorphism 
\[
\Hn: \cM _t (G, \cF_G ) \to D^{\Omega} (G):[X]-[Y]\mapsto[\widetilde H_n (X; k) ]-[\widetilde H_m (Y; k)],
\]
%$$\Hn: \cM _t (G, \cF_G ) \to D^{\Omega} (G)$$ by declaring $\Hn ([X]-[Y])=[\widetilde H_n (X; k) ]-[\widetilde H_m (Y; k)]$ 
where $n$ and $m$ are the dimensions of $X$ and $Y$.  Note that if $[X_1]-[Y_1] =[X_2]-[Y_2]$, then there is a tight Moore $G$-space $Z$ such that $$X_1\ast Y_2 \ast Z \cong X_2 \ast Y_1 \ast Z.$$  Therefore we have $\Hn ([X_1])-\Hn ([Y_1] ) =\Hn ([X_2]) - \Hn ([Y_2] ) $ in $D^{\Omega } (G)$. This shows that $\Hn$ is a well-defined homomorphism.
 
\begin{proposition}\label{pro:DependsOnly} There is a factorization $$\Hn= \Psi_G \circ \Dim$$
relating the Bouc homomorphism $\Psi_G:C(G,\cF_G)\to D^\Omega(G)$ to the maps $\Dim$ and $\Hn$ defined above.
%In particular, if $X$ is a finite, capped $\un{n}$-Moore $G$-space, then the equivalence class $[H_n (X; k)]$ in $D^{\Omega } (G)$ depends only on the function $\un{n}$.
\end{proposition}

\begin{proof} Follows from Theorem \ref{thm:Moore} and Lemma \ref{lem:DimSum}.
\end{proof}

%%-------------------

\section{Capped Moore $G$-spaces}\label{sect:CappedMoore}

For our applications we also need to consider Moore $G$-spaces with arbitrary isotropy. The unit spheres of real representations are important examples of  Moore $G$-spaces and in general they do not have $p$-subgroup isotropy. A real representation sphere may also have nonempty $S$-fixed-points.

Throughout this section we will work with the family $\cF_p$ of all $p$-subgroups of $G$. We denote the orbit category over $\cF_p$ by $\cO_p(G)$ to simplify the notation.

\begin{definition} Let $\un{n}: \cF_p \to \ZZ$ be a superclass function and $X$ an $\un{n}$-Moore $G$-space over $k$. Let $m:=\un{n} (S)$. We say $X$ is a \emph{capped} Moore space if the reduced homology  $\widetilde H_m ( X^S ; k)$, considered as an $\overline N_G(S)$-module, has a trivial component $k$. %the trivial module $k$ as a summand.
\end{definition}

Note that if $X$ is a Moore $G$-space whose point-stabilizers are all non-Sylow $p$-subgroups, then $X^S=\emptyset$. In this case we have 
$m=-1$ and $\widetilde H_m (X ^S ; k)\cong k$, so that $X$ is a capped Moore $G$-space.

The join of two capped Moore $G$-spaces is a capped Moore $G$-space, so the $G$-homotopy classes of capped (tight) Moore $G$-spaces form an additive monoid under the operation $[X]+[Y]:=[X\ast Y]$. Let $\cM_t (G, \cF_p )$ denote the Grothendieck group of $G$-homotopy classes of capped tight Moore $G$-spaces over the field $k$. Note that, while our definition of Moore $G$-space allows for non-$p$-group isotropy, we only require that $X^H$ be a Moore space when $H$ is a $p$-group.
 
The dimension function of a capped Moore $G$-space can be defined as in Definition \ref{def:DimFunctionMoore}, 
yielding the group homomorphism $$\Dim : \cM_t (G, \cF_p)\to C(G, p).$$
The composition of $\Dim$ with the Bouc homomorphism $\Psi_G : C(G, p) \to D^{\Omega} (G)$ gives the group homomorphism 
$$\Hn: \cM _t (G, \cF_p) \to D^{\Omega} (G).$$ The homology of a capped Moore space may not be a Dade $kG$-module, 
so it is not possible to explain this new homomorphism using the assignment $X\mapsto [H_n (X^?; k) ]$ as in the previous 
section. However it is still possible to give an interpretation of this homomorphism in terms of the homology of a Moore $G$-space. 
For this we first prove a lemma, which can be interpreted as a generalization of Lemma \ref{lem:ShortExact}.

\begin{lemma}\label{lem:ShortExact2}
Let $X$ be a $G$-set such that $X^S=\emptyset$, and let $0\to L\to k X\to N\to 0$ be an $X$-split 
short exact sequence of $kG$-modules. If $N$ is an endo-$p$-permutation module, then $L$ is as well.  
Moreover, if $N$ has a capped Dade module summand $N_1$, then $L$ has a summand $L_1$, 
also a capped Dade module, that satisfies $[L_1]=\Omega_X+[N_1]$ in $D(G)$.
\end{lemma}

\begin{proof} Repeating the argument in Lemma \ref{lem:ShortExact} we see that 
\[
\End(L)\oplus(k X\otimes N)\oplus(k X\otimes N^\ast)\cong\End(N)\oplus\End(k X).
\]
As $N$ is an endo-$p$-permutation module, $\End(L)$ is 
a $p$-permutation module, and thus $L$ is an endo-$p$-permutation module.

Assume now that  $N \cong N_1 \oplus N_2$ where $N_1$ is a  capped Dade $kG$-module.  
Applying the Relative Schanuel's Lemma to the sequences
\[
\xymatrix@=1.5em{
0 \ar[r]&L\ar[r]&k X\ar[r]&N\ar[r]\ar@{=}[d]&0\\
0\ar[r]&\Delta (X)\otimes N\ar[r]&k X\otimes N\ar[r]&N\ar[r]&0
}
\]
we obtain
\[
L\oplus(k X\otimes N)\cong k X\oplus(\Delta(X)\otimes N) .
\]
Now $\Delta (X) \otimes N_1$ is a capped Dade module that appears as a summand on the right hand side, so
the cap of $\Delta (X) \otimes N_1$ is isomorphic to a component on the left hand side. 
Since all components of $kX \otimes N$ have non-Sylow vertices,
$cap (\Delta (X) \otimes N_1)$ must be isomorphic to a summand of $L$, say $L_1$. Then $L_1$ is 
a capped Dade module such that $$[L_1]=[\Delta (X) \otimes N_1]=\Omega _X + [N_1],$$ 
completing the proof.
\end{proof}

%Before we state our next result we need to introduce more terminology. 

In the proof of the next proposition we make use of the Brauer quotient construction, which we now recall.
Given a $kG$-module $M$ and a $p$-subgroup
$P \leq G$, we denote by $M^P$ the submodule of $P$-fixed elements in $M$. For every $Q\leq P$, we have the relative trace map
$tr_Q^P : M^Q \to M^P$ defined by $m \mapsto \sum _{xQ\in P/Q} x\cdot m$. The quotient 
$$M[P]:=M^P/ \sum _{Q<P } tr _Q ^P (M^Q)$$
is an $\overline N_G(P)$-module, called the \emph{Brauer quotient of $M$ at $P$}.
The quotient map $M^P \to M[P]$ will be written $Br_P$.
It is easy to see that the Brauer quotient of the permutation module $k[X]$ at $P$ is isomorphic to $k[X^P]$ as a $k[\overline N_G(P)]$-module (see \cite[1.1]{Broue}). %For simplicity of notation, we write $\overline N_G(P)$ for the quotient group $N_G(P)/P$.

Bouc \cite{Bouc-Resolutions} studied the connection between Brauer quotients and modules over the orbit category $\cO_p (G)$. Let $M$ be a $k\cO_p (G)$-module. For every $p$-subgroup $P \leq G$, 
there is a restriction homomorphism $\res^P _1: M(P) \to M(1)$,  which commutes with the maps induced by $G$-conjugation.
Since $P$ acts trivially on $M(P)$, this implies that $\res^P _1 (M(P))\subseteq M(1)^P$, and thus we can compose with the quotient map $Br_P$ to obtain the homomorphism
\[
\brres_P : M(P) \xrightarrow{\res^P_1} M(1) ^P  \xrightarrow{Br_P} M(1)[P].
\]
We will need the following result due to Bouc.

\begin{proposition}[\cite{Bouc-Resolutions}, Proposition 6.5]\label{pro:BoucResolutions}  A $k\cO_p(G)$-module $M$ has a finite projective 
resolution if and only if $M(1)$  is a $p$-permutation $kG$-module and the map $\brres_P: M(P) \to M(1)[P]$
is an isomorphism of $k\overline N_G(P)$-modules for every $p$-subgroup $P \leq G$. 
\end{proposition}
 
Brauer quotients interact nicely with the Green correspondence: If $M$ is an indecomposable  
$p$-permutation module with vertex $P$, then the Brauer quotient $M[P]$ is the indecomposable 
projective $k[\overline N_G(P)]$-module corresponding to $M$ under the Green correspondence 
(see \cite[Theorem 3.2]{Broue}). %Using these and the lemma above we can prove the following.

We are now ready to state the main result of this section.

\begin{theorem}\label{thm:CappedMooreHom} Let $\un{n}: \cF_p \to \ZZ$ be a superclass function and $X$ an $n$-dimensional capped $\un{n}$-Moore $G$-space.  Assume 
also that $X$ is full and tight. Then the reduced homology
$\widetilde H_n (X; k)$ has a summand $M$ that is a capped Dade $G$-module satisfying $$[M]=\Psi _G (\Dim (X))$$ in $D^{\Omega } (G)$.
\end{theorem}

\begin{proof} Let $m :=\un{n} (S)$, and let 
$$ \widetilde C_* (X^?; k):  0 \to k[X_n ^?] \maprt{\partial _n} k[X_{n-1} ^?] \to \cdots \to k[X_m ^? ] \maprt{\partial _m} \cdots \to k[X_0 ^?] \to \underline{k} \to 0$$
be the reduced chain complex of $X$ over the orbit category $\cO_p(G)$. Note that this chain complex may not be a chain complex of free $k\cO_p$-modules since the isotropy subgroups of $X$ are not assumed to be $p$-subgroups of $G$. 

Consider the truncation of this complex at dimension $m$:
$$0 \to k[X_m ^? ] \maprt{\partial _m} \cdots \to k[X_0 ^?] \to \underline{k} \to 0.$$
Since $\un{n} (P) \geq \un{n}(S)=m$ for every $p$-subgroup $P$, this complex has 
no homology in dimensions  $i<m$. Let $K_m=\ker \partial _m$. The constant module $\un{k}$ and the modules $k[X_i^?]$ have finite 
projective resolutions (see \cite[Cor 3.14, 3.15]{HPY}), hence by an easy induction we can conclude that 
$K_m$ has a finite projective resolution. Then by Proposition \ref{pro:BoucResolutions}, $K_m (1)$ is a $p$-permutation module and 
for every $p$-subgroup $P$, 
$$\brres_P : K_m (P) \to K_m (1) [P]$$ is an isomorphism of  $k[\overline N_G(P)]$-modules. 
 
Since $X_i ^S=\emptyset $ for every $i>m$, we have $K_m ( S) \cong \widetilde H_m (X^S; k)$, and hence $k \mid K_m (S)$. This means that $K_m (1)[S]$ also has $k$ as  a summand. Let $U$ be a component of $K_m (1) $ such that $U[S] \cong k$. Then $U$ has vertex $S$, and by \cite[Theorem 3.2]{Broue} $U$ corresponds to $k$ under the Green correspondence. Thus $U \cong k$ by the uniqueness of the Green correspondent. We conclude that $k \mid K_m (1)$. 
 
By applying Lemma \ref{lem:ShortExact2} inductively, we see that $\widetilde H_n (X; k)$ has a summand $M$ that is a capped Dade module satisfying 
$[M]=\sum _{i=m+1} ^n \Omega _{X_i} $. Since $X_i ^S \neq \emptyset $ for all $i \leq m$,  
by definition we have $\Omega _{X_i}=0$ in $D^{\Omega } (G)$ for all $i\leq m$. Hence we obtain 
\[
[M]=\sum _{i=1} ^n \Omega _{X_i}=\Psi _G \left(\sum _{i=1} ^n  \omega _{X_i}  \right)=\Psi _G (\Dim (X) ).
\]
%$$[M]=\sum _{i=1} ^n \Omega _{X_i}=\Psi _G (\sum _{i=1} ^n  \omega _{X_i}  )=\Psi _G (\Dim (X) ).$$ 
Note that the last equality holds because $X$ is full. This completes the proof.
\end{proof}

Now we are ready to give a proof for Theorem \ref{thm:IntroInclusion}.

\begin{proof}[Proof of Theorem \ref{thm:IntroInclusion}] Assume that $f= \dim (V)-\dim (W)$ 
for some $k$-oriented real representations $V$ and $W$ of $G$.  The unit sphere $X=S(V)$ can 
be triangulated to obtain a full $G$-CW-complex structure on $X$ (see \cite{Illman}). 
It is clear that $X$ is a tight Moore $G$-space. By the assumption of $k$-orientability, the 
$\overline N_G(S)$-action on $\widetilde H_* (X^S; k)$ is trivial, so $X$ is a capped Moore space.  
Hence, by Theorem \ref{thm:CappedMooreHom}, 
the reduced homology $\widetilde H_n (X; k)$ has a summand $M$ that is a capped Dade $G$-module satisfying 
$$[M]=\Psi _G (\Dim (X))$$ in $D^{\Omega } (G)$. Since  $\widetilde H_n (X; k)\cong k$ as a $kG$-module, 
$M$ is isomorphic to $k$. 
Hence  $$\Psi _G (\Dim (X))=[k]=0.$$ A similar conclusion also holds for the $G$-sphere $Y=S(W)$, so we obtain 
$$\Psi _G(f)=\Psi _G (\Dim (V) -\Dim (W))=\Psi_G ( \Dim(X) -\Dim (Y))=0,$$ and hence $f \in \ker \Psi_G$.
\end{proof}

\section{Oriented Artin-Borel-Smith functions}\label{sect:OrientedArtin}

In the case where $G=S$ is a $p$-group, the kernel of the Bouc homomorphism $\Psi_S: C(S) \to D^{\Omega} (S)$ is precisely the subgroup of superclass functions satisfying the  Borel-Smith conditions, which we now recall (see \cite[Thm. 1.2 and Def. 3.1]{BoucYalcin}). 
%or \cite[Def. 5.1]{tomDieck-Transformation}).

\begin{definition}
\label{def:BorelSmith}
Let $S$ be finite $p$-group. A superclass function $f \in C(S)=C(S, \cF_p)$ is called a \emph{Borel-Smith function} if it satisfies the following conditions:

\begin{enumerate}
\item\label{def:BorelSmith1} If $L \lhd K \leq S$, $K/L \cong \ZZ /p$, and $p$ is odd,
then $f(L)-f(K)$ is even.

 \item\label{def:BorelSmith2} If $L \lhd K \leq S$, $K/L \cong \ZZ /p \times \ZZ /p$, and $\{ K_i/L \}_{i=0} ^p$ are the subgroups of order $p$ in $K/L$, then
     \[f(L)-f(K)=\sum
_{i=0} ^p (f(K_i)-f(K)).\]

\item\label{def:BorelSmith3} If $L \lhd K \lhd N \leq N_S(L)$ and $K/L \cong \ZZ /2$,
then $f(L)-f(K)$ is even if $N/L \cong \ZZ /4$, and $f(L)-f(K)$ is divisible by $4$ if $N/L$ is the quaternion group $Q_8$ of order $8$.
\end{enumerate}
\end{definition}

It is known from classical Smith theory that if $X$ is a $G$-space whose mod-$p$ homology is isomorphic to the mod-$p$ homology of a sphere, then 
for every $p$-subgroup $P \leq G$, the fixed-point set $X^P$ also has the mod-$p$ homology of a sphere (see \cite[Thm. 4.23]{tomDieck-Transformation}). Let $\underline{n}(P)$ denote the dimension for which $\widetilde H_* (X^P; \FF _p ) \neq 0$. Then $\un{n}$ is a superclass function, which satisfies the Borel-Smith functions (see \cite[pg. 202-210]{tomDieck-Transformation}). We denote the group of Borel-Smith functions for a $p$-group $S$ by $C_b(S)$.   

\begin{definition} Let $G$ be a finite group and $S$ be a Sylow $p$-subgroup of $G$. 
A superclass function $f : \cF_p \to \ZZ$ is called a \emph{Borel-Smith function} if its restriction to $S$ is a Borel-Smith function.
We denote the subgroup of Borel-Smith functions in $C(G, p)$ by $C_b(G, p)$.
\end{definition}
 
The following is easy to observe. 

\begin{lemma}\label{lem:BorelSmith} The kernel of the Bouc homomorphism $\Psi_G : C(G, p) \to D^{\Omega} (G)$ lies inside $C_b(G, p)$.
\end{lemma}

\begin{proof}  By Lemma \ref{lem:Restrict}, there is a well-defined restriction homomorphism $\res^G _S : D (G) \to D(S)$, given by $\res^G _S [M] =[\res^G _S M]$ for every Dade $kG$-module $M$. It is easy to see from the definitions that $$\res^G _S ( \Omega _X)=\Omega _{\res^G _S X} $$ for every $G$-set $X$. There is also a homomorphism $\res^G _S : C(G, p) \to C(S)$ defined by restriction of  domain. We have
$$ \res ^G _S (\omega_X)=\omega _{\res^G _S X}$$ for every $G$-set $X$. This shows that  $\Psi_G$ commutes with the restriction map to $S$. If $f \in \ker \Psi_G$, then $\res^G _S f \in \ker \Psi _S$. By \cite[Thm. 1.2]{BoucYalcin}), $\ker \Psi _G = C_b(S)$, and hence $f \in C_b (G, p)$.
\end{proof}

In the case where $G=S$ is a $p$-group, the group of Borel-Smith functions $C_b(S)$ is exactly equal to the image of the dimension function $$\Dim : R_{\RR} (S) \to C(S)$$ from the real representation ring to the group of superclass functions (see \cite{DotzelHamrick}). For an arbitrary finite group $G$ there is a similar theorem due to Bauer \cite{Bauer}, stated in terms of dimension functions defined on prime-power order subgroups. Let $\cP$ denote the family of all subgroups of $G$ with prime-power order and let $C(G , \cP)$ denote the group of superclass functions $f: \cP \to \ZZ$. Bauer imposes the following additional condition on the superclass functions in $C(G, \cP)$ (see also \cite{ReehYalcin}).

\begin{definition}\label{def:ArtinCondition} A function $f \in C(G, \cP)$ \emph{satisfies the Artin condition} if:
\begin{itemize}
\item[($\ast$)] For any distinct prime numbers $p$ and $q$, consider $L \lhd K \lhd H \leq N_G(L) $  subgroups of $G$
such that  $K$ is a $p$-group, $K/L\cong \ZZ /p $, and $H/K \cong \ZZ/ q^r$. Then $f(L)\equiv f(K) \pmod{q^{r-l}}$, 
where $H/K$ acts on $K/L$ with kernel of order $q^l$.
\end{itemize}
\end{definition}

Bauer proves the following:

\begin{theorem}[Bauer \cite{Bauer}, Thm. 1.3]\label{thm:Brauer} Let $f \in C(G, \cP)$ be a superclass function that 
satisfies the Borel-Smith conditions when restricted to a Sylow subgroup, as well as the Artin condition $(\ast)$. 
Then there is a virtual real representation $x=[V]-[W]$ such that $f=\dim V-\dim W$.
\end{theorem}

If one studies the proof of Bauer's theorem it is easy to see that one needs the Artin condition  $(\ast)$ to hold only 
in the case when $K$ is a cyclic $p$-subgroup. This follows from Definition \ref{def:BorelSmith}\ref{def:BorelSmith2}, 
which implies that a Borel-Smith function is determined by its values on the cyclic $p$-subgroups. 
Also observe that when $p=2$, there is no nontrivial automorphism of $K/L \cong \ZZ/2$, so in this case the Artin condition 
always holds. 
 
Let $k$ be a field of characteristic $p$. The real representations in Bauer's theorem need not be $k$-orientable 
(see Definition \ref{def:kOrientable}) when $p$ is odd. The problem derives from a certain real representation of the dihedral group $D_{2p}$, whose representation sphere is not $k$-orientable, used in the induction of Bauer's proof.  We can modify Bauer's construction to obtain only $k$-orientable representations, 
at the cost of introducing a variation of the Artin condition. We name this condition only for functions 
defined on $\cF_p$, but it can easily be extended to families defined on $\cP$ if necessary.
 
\begin{definition}\label{def:OrientedArtin} A function $f \in C(G, p)$  \emph{satisfies the oriented Artin condition} if: 
\begin{itemize}
\item[($\ast\ast$)]  For any distinct prime numbers $p$ and $q$, consider $L\triangleleft K\triangleleft H\leq N_G(L)$ subgroups of $G$ such that $K$ is a cyclic $p$-group, $K/L\cong \ZZ/p$, and $H/K\cong\ZZ/q^r$.  Then $f(L)\equiv f(K) \pmod{ 2q^{r-l}}$, where $H/K$ acts on $K/L$ with kernel of order $q^l$.
\end{itemize}
\end{definition}

Note that the oriented Artin condition ($\ast\ast$) differs from the ordinary Artin condition ($\ast$) not only in the restriction of our attention to the cyclic $p$-subgroups $H$, but also in the extra factor of $2$ in the modular equation.  This factor is necessary to obtain $k$-orientable representations. 

\begin{definition} An \emph{oriented Artin-Borel-Smith function} is a superclass function that satisfies the Borel-Smith 
conditions of Definition \ref{def:BorelSmith} and the oriented Artin condition of Definition \ref{def:OrientedArtin}. 
The subgroup of $C(G,p)$ consisting of the oriented Artin-Borel-Smith functions is denoted by $C_{ba^+} (G, p)$.
\end{definition}

Let $R_{\RR} ^{+} (G, k)$ denote the Grothendieck group of $k$-oriented real representations  of $G$.
The following is easy to prove using Bauer's construction in \cite[Thm 1.3]{Bauer} with small variations.

\begin{theorem}\label{thm:OriBauer} The image of the dimension function $R^+ _{\RR} (G, k) \to C(G, p)$ is equal to $C_{ba^+} (G, p)$.
\end{theorem}

\begin{proof} The argument given in \cite[Proposition 1.2]{Bauer} directly applies here to give that the image of the dimension function 
lies in $C_{ba^+} (G)$. Let $L \triangleleft K \triangleleft H \leq N_G(L)$ be as in Definition \ref{def:OrientedArtin}. If we take $m=f(L)$ and $n=f(K)$ and repeat Bauer's equivariant cohomology argument, we obtain that the periodicity of the mod-$p$ cohomology of the group $G$ divides $m-n$. 
The $p$-period of a group is computed by Swan in \cite[Theorem 1 and 2]{Swan-Period}. According to these calculations, 
 the $p$-period of  $H$ is $2q^{r-l}$.
 
Conversely, if $P \leq G$ is a non-cyclic $p$-subgroup then there is a normal subgroup $N \trianglelefteq P$ 
such that $P/N \cong \ZZ/p\times \ZZ /p$. Condition \ref{def:BorelSmith2} of Definition \ref{def:BorelSmith}
then implies that the value of a Borel-Smith function $f$ at $P$ is determined by its  values on proper subgroups $Q <P$.  
This implies that 
a function in $f\in C_{ba^+} (G, p)$ is uniquely determined by its values on the cyclic $p$-subgroups of $G$. 

Take $f \in C_{ba^+} (\cF)$. We will show that there is a virtual representation $x\in R_{\RR} (G)$ such that $\Dim (x) (P)=f(P)$ 
for all cyclic subgroups $P \leq S$. The argument is by induction, so we introduce some terminology for the intermediate step: 
Given a family of cyclic $p$-subgroups $\cH$, we say  $f$ is \emph{realized over $\cH$} 
if there is a virtual $k$-oriented $G$-representation $x$ such that $\Dim (x)  (P)=f(P)$ for every  $P \in \cH$.

Any $f$ is realizable at the trivial subgroup $1 \leq S$ by taking $f(1)$ copies of the trivial representation $\RR$.  
Suppose that $f$ is realized over some nonempty family $\cK$ of cyclic subgroups of $S$. Let $K\leq S$ be a cyclic subgroup all of whose proper subgroups are contained in $\cK$, and let $\cK'$ be the family obtained from $\cK$ by adding the conjugacy class of $K$.
We will show that $f$ is also realizable over $\cK'$. By induction this will give us the realizability of $f$ over all cyclic subgroups.

Since $f$ is realizable over $\cK$, there is an element $x\in R^{+}_{\RR} (G)$ such that $\Dim (x) (J)=f(J)$ for every $J \in \cK$. By replacing $f$ with $f-\Dim (x)$, we may assume that $f(J)=0$ for every $J\in \cK$. To prove that $f$ is realizable over the larger family $\cK'$, we will show that for every prime $q$, there is an integer $n_q$ coprime to $q$ such that $n_qf$ is realizable over $\cK'$ by some virtual representation $x_q \in R_{\RR } (\cF)$. This will be enough to complete our proof  by a simple number theory argument (see \cite[Theorem 8.7]{ReehYalcin} for details).

If $q=p$, then by earlier results due to tom Dieck \cite{tomDieck-Transformation} and Dotzel-Hamrick \cite{DotzelHamrick}, there is an integer $n_p$, coprime to $p$, such that $n_p f$ is realized by an element in $x_p \in R^{+}_{\RR} (G)$ (see also \cite[Proposition 4.9]{ReehYalcin}). Thus we may assume that $q \neq p$.  

Let $L\leq K$ be the unique subgroup of $K$ of index $p$. Choose $H \leq N_G(K)$ such that $H/K$ is a Sylow $q$-subgroup of $\overline N_G(K)$. Consider the homomorphism $\rho: H \to \Aut (K)$ determined by the $H$-conjugation action on $K$, and set  $R=\ker \rho$. Note that $R$ is nilpotent and hence isomorphic to the product $K \times R_q$, where $R_q$ is the unique Sylow $q$-subgroup of $R$. Since $K$ is cyclic and $p$ is odd, $\Aut (K)$ is cyclic. Thus $H/R$ is a cyclic $q$-group, say of order $q^t$. There is an element $h \in H$ such that $\langle h \rangle \cdot R=H$. The oriented Artin condition for the subgroups $L \leq K \leq  \langle h \rangle \cdot R$  gives that $2q^t \mid (f(L)-f(K))$.  

Let $\overline U$ be the real $K$-representation defined by $\overline U =2\RR -W$, where $W$ is the $2$-dimensional real representation of $K$ with kernel $L$ defined by sending a generator of $K/L$ to the rotation by $2\pi/p$. Note that $\overline U$ is a $k$-orientable representation of $K$. 
We can consider $\overline U$ as a representation of $R$ via the projection map $R \to K$. Let $U=\ind _{R} ^G \overline U$. It is easy to see that $\dim (U^J)=0$ for any $J\in \cH$. If $J \in \cH'-\cH$, then $J$ is conjugate to $K$ and we have 
$$\dim (U^J)= \dim (\overline U ) \cdot |N_G (K): K| =2 |N_G(K): H| \cdot  |H: K|=2 n_q  q^t,$$
where $n_q:=|N_G(K): H|$ is coprime to $q$. If we take $V_q$ as $f(H)/ 2q^t$ copies of $U$, then $V_q$ realizes $n_q f$.
This completes the proof.
\end{proof}

As we pointed out above, when $p=2$, both the Artin condition of Definition \ref{def:ArtinCondition}
and the oriented Artin condition of Definition \ref{def:OrientedArtin} hold trivially. So in this case we have $$C_{ba^+} (G, p)=C_{ba} (G, p)=C_b(G, p).$$
Note also that when $k$ is of characteristic $2$, every real representation is $k$-orientable, so $R_{\RR} ^+ (G, k) =R_{\RR} (G)$. Thus when $p=2$, the conclusion of Theorem \ref{thm:OriBauer}  coincides with that of Theorem \ref{thm:Brauer}. Thus $C_{ba^+} (G, p)$ is equal to the image of the dimension function $\Dim : R_{\RR} (G) \to C (G,p)$,  which is also equal to the group of Borel-Smith functions $C_b(G, p)$.  

When $p$ is odd,  for any $L\trianglelefteq H$ with $H/L \cong \ZZ/p$, a Borel-Smith  function $f$ will satisfy the modular equation
 $f(L)\equiv f(H) \pmod{2} $.
In particular if $f \in C_{ba^+}(G, p)$ is such that $f(1)$ is even, then $f(P)$ is even for all $P \in \cF_p$. We call such functions \emph{even-valued}. 
In the proof  of Theorem \ref{thm:OriBauer} 
we used the representation $\overline U$, a restriction of a complex representations to $\RR$. From  this one can see that an even-valued oriented Artin-Borel-Smith function $f \in C_{ba^+} (G, p)$ can be realized by the (real) dimension function of a complex representation of $G$. We conclude the following.

\begin{proposition} Let $\KK$ be the field such that $\KK=\RR$ if $p=2$ and $\KK=\CC$ if $p>2$. Let $f \in C_{ba^+} (G, p)$ be such that $f(1)$ is even when $p$ is odd. Then $f$ is in the image of the dimension function $$\Dim : R_{\KK} (G) \to C_{ba^+ } (G, p)$$
defined by $(\Dim V)(P)=\dim _{\RR} (V^P)$ for every $p$-subgroup $P$ of $G$.
\end{proposition}

\begin{proof} The case $p=2$ is clear from the discussion above. Let $p>2$ and assume that $f(1)$ is even. We started the induction in the proof of Theorem \ref{thm:OriBauer} by subtracting $f(1)$ copies of the dimension function of the constant real representation from $f$. Since $f(1)$ is even, this step can be achieved using $f(1)/2$ copies of the constant complex representation. 
In the inductive steps of the proof, we use the representation $\overline U =2\RR-W$. Both $2\RR$ and $W$ come from complex representations, so $\ind_R ^G \overline U \in R_{\CC} (G)$.
\end{proof}

As a consequence of Lemma \ref{lem:BorelSmith} and Theorem \ref{thm:OriBauer}, we have the following:

\begin{theorem} Let $G$ be a finite group and $\Psi_G: C(G, p)\to D^{\Omega} (G)$ the Bouc homomorphism. Then $$ C_{ba^+} (G, p) \subseteq \ker \Psi _G \subseteq C_{b} (G, p).$$ In particular, when $p=2$, we have $\ker \Psi_G=C_{ba+} (G, p)=C_{b} (G, p)$.
\end{theorem}

It is interesting to ask whether the equality $\ker \Psi_G=C_{ba^+ } (G, p)$ also holds when $p$ is an odd prime; we were not able to find any counter-examples to this claim. The following computation gives some evidence for the relevance of the oriented Artin condition:

\begin{lemma}\label{lem:SemiDirect} Let $Q \trianglelefteq G$ be a cyclic subgroup of $G$ of order $p$ such that $G/Q \cong \ZZ/q^r$. Suppose that $G/Q$ acts on $Q$ with kernel order $q^l$. Then, $D^{\Omega } (G) \cong \ZZ / 2q^{r-l}$ and the equality $\ker \Psi _G = C_{ba^+} (G, p)$ holds.
\end{lemma}

\begin{proof}  Since $Q$ is a Sylow $p$-subgroup of $G$, we have $D^{\Omega } (G)$ is generated by $\Omega _{G/1}$.
The cohomology of the group $G$ with $k$ coefficients is periodic, hence there exists an integer $n$ such that 
$\Delta (G/1) ^{\otimes n} \cong k\oplus (proj)$. 
The smallest such $n$ is equal to the $p$-period of $G$. By Swan \cite[Thm 1 and 2]{Swan-Period}, the $p$-period of  
$G$ is $2q^{r-l}$.
 
Note that $[k\oplus (proj)]=0$ in $D^{\Omega } (G)$. Conversely if $M$ is a Dade $kG$-module such that $[M]=0$ in $D^{\Omega } (G)$, then $M\cong k\oplus (proj)$. We conclude that the smallest $n$ such that $\Delta (G/1) ^{\otimes n} \cong k \oplus (proj)$ is equal to the exponent of $\Omega _{G/1}$. Hence the exponent of $\Omega _{G/1}$ is $q^{r-l}$.  Thus $D^{\Omega } (G)\cong \ZZ/ 2q^{r-l}$.

If $f = a_1 \omega _{G/1} + a_Q  \omega _{G/Q} $ is in the kernel of  $\Psi _G$, then $a_1 \Omega _{G/1}=0$
giving that  $2q^{r-l}$ divides $a_1$. Thus $a_1=f(1)-f(Q)\equiv 0 \pmod{2q^{r-l}}$, and hence $f\in C_{ba^+} (G, p)$.
\end{proof}

\end{document}